\pdfoutput=1
\RequirePackage{ifpdf}
\ifpdf % We are running pdfTeX in pdf mode
\documentclass[pdftex]{sigma}
\else
\documentclass{sigma}
\fi

\usepackage{mathtools}
\usepackage[arrow,curve,matrix,tips,2cell]{xy}

\numberwithin{equation}{section}

\numberwithin{theorem}{section}
\newtheorem{mtheorem}[theorem]{Main Theorem}
\newtheorem{problem}[theorem]{Problem}
\newtheorem{Corollary}[theorem]{Corollary}
\newtheorem{Lemma}[theorem]{Lemma}
\newtheorem{Proposition}[theorem]{Proposition}
 { \theoremstyle{definition}

\newtheorem{Remark}[theorem]{Remark} }

\newcommand{\R}{\mathbb{R}}

\newcommand{\N}{\mathbb{N}}
\newcommand{\Z}{\mathbb{Z}}
\newcommand{\F}{\mathbb{F}}
\newcommand{\B}{\mathrm{B}}
\newcommand{\E}{\mathrm{E}}

\newcommand\Sym{\mathfrak S}

\newcommand{\FF}{\operatorname{F}}
\newcommand{\GG}{\operatorname{G}}

\newcommand{\cat}{\operatorname{cat}}

\newcommand{\id}{\operatorname{id}}

\newcommand{\wgt}{\operatorname{wgt}}

\begin{document}

\allowdisplaybreaks

\newcommand{\arXivNumber}{1712.07930}

\renewcommand{\thefootnote}{}

\renewcommand{\PaperNumber}{022}

\FirstPageHeading

\ShortArticleName{Counting Periodic Trajectories of Finsler Billiards}

\ArticleName{Counting Periodic Trajectories of Finsler Billiards\footnote{This paper is a~contribution to the Special Issue on Algebra, Topology, and Dynamics in Interaction in honor of Dmitry Fuchs. The full collection is available at \href{https://www.emis.de/journals/SIGMA/Fuchs.html}{https://www.emis.de/journals/SIGMA/Fuchs.html}}}

\Author{Pavle V.M.~BLAGOJEVI\'{C}~$^{\dag^1\dag^2}$, Michael HARRISON~$^{\dag^3}$, Serge TABACHNIKOV~$^{\dag^4}$\newline and G\"unter M. ZIEGLER~$^{\dag^1}$}

\AuthorNameForHeading{P.V.M.~Blagojevi\'{c}, M.~Harrison, S.~Tabachnikov and G.M.~Ziegler}

\Address{$^{\dag^1}$~Institut f\"ur Mathematik, FU Berlin, Arnimallee 2, 14195 Berlin, Germany}
\EmailDD{\href{mailto:blagojevic@math.fu-berlin.de}{blagojevic@math.fu-berlin.de}, \href{mailto:ziegler@math.fu-berlin.de}{ziegler@math.fu-berlin.de}}
\URLaddressDD{\url{http://page.mi.fu-berlin.de/blagojevic/}\newline
\hspace*{15mm}\url{http://www.mi.fu-berlin.de/math/groups/discgeom/ziegler}}

\Address{$^{\dag^2}$~Mathematical Institut SASA, Knez Mihailova 36, 11000 Beograd, Serbia}

\Address{$^{\dag^3}$~Department of Mathematical Science, Carnegie Mellon University,\\
\hphantom{$^{\dag^3}$}~Pittsburgh, PA 15213, USA}
\EmailDD{\href{mailto:mah5044@gmail.com}{mah5044@gmail.com}}
\URLaddressDD{\url{https://sites.google.com/site/michaelharrisonmath/}}

\Address{$^{\dag^4}$~Department of Mathematics, Pennsylvania State University, University Park, PA 16802, USA}
\EmailDD{\href{mailto:tabachni@math.psu.edu}{tabachni@math.psu.edu}}
\URLaddressDD{\url{http://www.personal.psu.edu/sot2/}}

\ArticleDates{Received September 11, 2019, in final form March 25, 2020; Published online April 03, 2020}

\Abstract{We provide lower bounds on the number of periodic Finsler billiard trajectories inside a quadratically convex smooth closed hypersurface $M$ in a $d$-dimensional Finsler space with possibly irreversible Finsler metric. An example of such a system is a billiard in a sufficiently weak magnetic field. The $r$-periodic Finsler billiard trajectories correspond to $r$-gons inscribed in $M$ and having extremal Finsler length. The cyclic group $\Z_r$ acts on these extremal polygons, and one counts the $\Z_r$-orbits. Using Morse and Lusternik--Schnirelmann theories, we prove that if $r\ge 3$ is prime, then the number of $r$-periodic Finsler billiard trajectories is not less than $(r-1)(d-2)+1$. We also give stronger lower bounds when $M$ is in general position. The problem of estimating the number of periodic billiard trajectories from below goes back to Birkhoff. Our work extends to the Finsler setting the results previously obtained for Euclidean billiards by Babenko, Farber, Tabachnikov, and Karasev. }

\Keywords{mathematical billiards; Finsler manifolds; magnetic billiards; Morse and Lusternik--Schnirelmann theories; unlabeled cyclic configuration spaces}

\Classification{37J45; 55R80; 70H12}

\begin{flushright}
\begin{minipage}{65mm}
\it Dedicated to D.~Fuchs on the occasion\\ of his 80th anniversary
\end{minipage}
\end{flushright}

\renewcommand{\thefootnote}{\arabic{footnote}}
\setcounter{footnote}{0}

\section{Introduction}

The study of mathematical billiards goes back to G.D.~Birkhoff who wrote in~\cite{Birkhoff1927}:
\begin{quote}
\dots\ in this problem the formal side, usually so formidable in dynamics, almost completely disappears, and only the interesting qualitative questions need to be conside\-red.
\end{quote}
One of the main motivations for the study of billiards has been their relation with mathematical physics and statistical mechanics, namely, with the Boltzmann ergodic hypothesis. We refer to the books \cite{Chernov-Markarian,Kozlov-Treshchev,Tabachnikov1995,Tabachnikov2005} for various aspects of mathematical billiards.

\subsection{Billiards in Euclidean geometry}

A Birkhoff billiard table is bounded by a smooth strictly convex closed hypersurface $M$ in $\R^d$. The billiard dynamical system describes the motion of a free particle inside $M$ with elastic reflection off the boundary. That is, the point (billiard ball) moves with unit speed along a~straight line until it hits the boundary~$M$; at the impact point, the normal component of the velocity instantaneously changes sign, while the tangential component remains the same, and the point continues its rectilinear motion. In dimension two, this is the familiar law of geometrical optics: the angle of incidence equals the angle of reflection.

An $r$-periodic billiard trajectory in $M$ is an $r$-tuple of points $(x_1,\ldots,x_r)$ of $M$ such that $x_i \neq x_{i+1}$ and the billiard reflection in $M$ takes segment $x_{i-1} x_i$ to $x_i x_{i+1}$ for $i=1,\ldots,n$ (as usual, the indices are understood cyclically, that is, $x_{r+i}=x_i$).

The Dihedral group of symmetries of a regular $r$-gon $D_{r}$ acts naturally on the set of all periodic billiard trajectories of period $r$ by cyclically permuting the points and reversing the orientation. Thus, when counting periodic orbits, it is natural to count such dihedral orbits.

\begin{problem}\label{problem : number of periodic in Euclidean case}
Let $d\geq 2$ and $r\geq 2$ be integers, and let $M^{d-1}\subseteq\R^d$ be a smooth closed strictly convex hypersurface. Estimate below the number $N_E\big(M^{d-1},r\big)$ of equivalence classes of periodic billiard trajectories of period $r$ inside $M^{d-1}$ modulo the action of the dihedral group $D_{r}$.
\end{problem}

The way we have formulated the problem, multiple trajectories are included into the count, so that for example a 2-periodic trajectory traversed thrice and a 3-periodic trajectory traversed twice, contribute to the number of 6-periodic trajectories. Of course, this issue is not a concern if $r$ is a prime.

The first progress in addressing the question posed in Problem \ref{problem : number of periodic in Euclidean case} was made by Birkhoff in 1927~\cite{Birkhoff1927}. He considered the case $d=2$, and proved that there exist at least two $r$-periodic orbits with every rotation number coprime with $r$, which implies that $N_E\big(M^{1},r\big)\ge 2\phi\big(\big\lfloor \tfrac{r}{2}\big\rfloor\big)$, where $\phi$ is the Euler totient function. We remark that this lower bound holds for the number of {prime} periodic trajectories. Birkhoff deduced his result from Poincar\'e's geometric theorem, that Poincar\'e published without proof shortly before his death and that Birkhoff proved a year later.

The concept of rotation number is not available in dimensions $d\geq 3$, but one can use the variational approach to the problem. Periodic billiard trajectories correspond to the critical points of the length function on $r$-tuple of points $(x_1,\ldots,x_r)$, that is, on $r$-gons inscribed in~$M$:
\begin{gather} \label{genfunct}
L(x_1,\ldots,x_r) = \sum_{i=1}^r |x_i - x_{i+1}|.
\end{gather}
This function is $D_{r}$-invariant.

The first to address Problem \ref{problem : number of periodic in Euclidean case} for $d=3$ was Babenko \cite{Babenko1990}, whose approach was based on analyzing critical points of the length function~$L$. Although his paper contained an error, the main idea was rescued and refined by Farber and Tabachnikov who established several lower bounds for $N_E\big(M^{d-1},r\big)$.

When dealing with critical points of functions, one applies either the Morse theory or the Lusternik--Schnirelmann theory.
The former usually gives stronger lower bounds on the number of critical points, but it applies to a more restrictive set of smooth functions, namely, Morse (or Morse--Bott) functions; the latter sometimes gives weaker lower bounds on the number of critical points, but it does not rely on genericity assumptions on the functions involved.

The variational methods subsequently produced a number of improved lower bounds in a~sequence of papers by a~number of authors:
\begin{itemize}\itemsep=0pt

\item {\em Farber \& Tabachnikov} \cite[Theorem~1(B), p.~555]{Farber-Tabachnikov2002-01} and \cite[Theorem~3]{Farber-Tabachnikov2002-02}.
		Let $d\geq 3$ be an integer, let $r\geq 3$ be an odd integer, and let $M^{d-1}$ be a generic smooth closed strictly convex hypersurface.
		Then
		\[N_{E}\big(M^{d-1},r\big)\geq (r-1)(d-1).\]
This is a Morse-theoretical result, that is why $M$ is assumed to be generic. The next results are Lusternik--Schnirelmann-theoretical, and they hold without the genericity assumption.
		
\item {\em Farber \& Tabachnikov} \cite[Theorem~1(A), p.~554]{Farber-Tabachnikov2002-02}.
	 Let $d\geq 4$ be an integer, let $r\geq 3$ be an odd integer, and let $M^{d-1}$ be a smooth closed strictly convex hypersurface.
	 Then
	 \[N_{E}\big(M^{d-1},r\big)\geq \lfloor\log_2 (r-1)\rfloor +d-1.\]

\item {\em Farber} \cite[Theorem~2, p.~589]{Farber2002-II}.
	Let $r\geq 3$ be an odd prime, and let $M^{d-1}$ be a smooth closed strictly convex hypersurface.
	
\begin{itemize}\itemsep=0pt
	\item If $d\geq 4$ is even, then \[N_{E}\big(M^{d-1},r\big)\geq r.\]
	\item If $d\geq 3$ is odd, then \[N_{E}\big(M^{d-1},r\big)\geq \frac{r+1}{2}.\]
	\end{itemize}

\item {\em Karasev} \cite[Theorem~1, p.~424]{Karasev2009}.
 Let $d\geq 3$, let $r\geq 3$ be prime, and let $M^{d-1}$ be a smooth closed strictly convex hypersurface.
 Then
		\[N_{E}\big(M^{d-1},r\big)\geq (r-1)(d-2)+2.\]
\end{itemize}

Let us also mention Mazzucchelli's paper \cite{Mazzucchelli} concerning the multiplicity of Birkhoff billiard periodic trajectories whose period is a power of a fixed prime number.

\subsection{Billiards in Finsler geometry}\label{BilFin}

Our goal is to extend the above described results to periodic trajectories in Finsler billiards.

From the point of view of physics, Finsler geometry describes the propagation of light in a~medium that is not necessarily homogeneous or isotropic. The speed of light depends on the point and the direction, and is given by a smoothly varying norm on the tangent spaces to the medium thought of as a smooth manifold. We allow these norms to be asymmetric.
These norms need not correspond to inner products, which is the case when the metric is Riemannian. To quote Chern's description~\cite{Chern1996}, ``Finsler geometry is just Riemannian geometry without the quadratic restriction''.

The distance $f(A,B)$ between points $A$ and $B$ is defined as the least time it takes light to travel from $A$ to $B$; in general, $f(A,B) \neq f(B,A)$. The trajectories of light are Finsler geodesics.
See Section \ref{Fgeo} for precise definitions.

\looseness=-1 An example of a Finsler manifold is a Minkowski space, that is, a finite-dimensional normed space. Another example is a projective metric in a domain in the projective space, a Finsler metric whose geodesics are straight lines. Hilbert's fourth problem asked to describe all projective Finsler metrics (a projective Riemannian metric is a metric of constant curvature), see, e.g.,~\cite{Hilbert}.

The first steps in the study of billiards in Finsler geometry were made by Gutkin and Tabachnikov in~\cite{Gutkin2002}. In this paper, the Finsler billiard reflection is defined (see Section~\ref{sec:reflectionlaw} below) using the variational approach, and Minkowski billiards are studied in some detail. Minkowski billiards have close relations with convex geometry (Mahler's conjecture) and symplectic topology (symplectic capacities), see \cite{Artstein2,Artstein1}.

We consider a domain in a Finsler manifold bounded by a smooth closed hypersurface $M^{d-1}$, a Finsler billiard table. We assume that $M$ is {\it strictly convex} in the following sense:
\begin{itemize}\itemsep=0pt
\item for every pair of points $x$, $y$ inside $M$, there is a unique geodesic from $x$ to $y$, and a unique geodesic from $y$ to $x$, both contained inside $M$ and distance-minimizing;
\item $M$ is quadratically convex: every geodesic tangent to~$M$ has second order contact with~$M$, and not higher order.
\end{itemize}
In particular, these conditions imply that~$M$ is diffeomorphic to the sphere $S^{d-1}$. The convexity is an open condition. In the Minkowski setting, this condition is the same as in Euclidean space;
see the end of Section~\ref{magnetic} for the geometric meaning of convexity in the case of planar magnetic billiards.

We are interested in periodic Finsler billiard trajectories inside $M$. They still correspond to critical points of the analog of the length function~(\ref{genfunct})
\[
\Lambda(x_1,\ldots,x_r) = f(x_1,x_2) + f(x_2,x_3) + \cdots + f(x_r,x_1),
\]
however this function has less symmetry than in the Euclidean case: it is invariant under the cyclic permutations of the points, but not under the orientation reversal. Thus $\Lambda(x_1,\ldots,x_r)$ is $\Z_r$-invariant, but not necessarily $D_{r}$-invariant.
Accordingly, when counting periodic Finsler billiard trajectories, we count $\Z_r$-orbits.

Denote by $N_F\big(M^{d-1},r\big)$ the number of equivalence classes of periodic Finsler billiard trajectories of period $r$ inside $M^{d-1}$ modulo the action of the cyclic group $\Z_r$.

\subsection{Statement of main results}

Our main result is as follows.

\begin{mtheorem} \label{mainthm}
Let $d\geq 3$ be an integer and $r\geq 3$ be a prime. Consider Finsler billiard inside a smooth closed hypersurface $M^{d-1}$, satisfying the above formulated strict convexity assumptions. Then
\begin{enumerate}\itemsep=0pt
\item[\rm (A)] $N_F\big(M^{d-1},r\big) \ge (r-1)(d-2)+1$.

\item[\rm (B)] For a generic $M$,
\begin{enumerate}\itemsep=0pt
\item[\rm (1)] if $d$ is even, then $N_F\big(M^{d-1},r\big) \ge (r-1)d$;
\item[\rm (2)] if $d$ is odd, then $N_F\big(M^{d-1},r\big) \ge(r-1)(d-1)$.
\end{enumerate}
\end{enumerate}
\end{mtheorem}

The general position assumption in case (B) is the assumption that the length $\Lambda(x_1,\ldots,x_r)$ is a Morse function.
The latter condition is generic in the sense that it holds in an open dense subset of strictly convex hypersurfaces, considered in the Whitney $C^{\infty}$ topology.
For Euclidean billiards, this is deduced in \cite[Lemma~4.4]{Farber-Tabachnikov2002-01} from an appropriate version of the multi-jet transversality theorem. A similar argument works in the Finsler case; we do not elaborate on it here.

\begin{Remark}The rate of growth of the numbers $N_F\big(M^{d-1},r\big)$, provided by Theorem \ref{mainthm}, is the same as in the above described results for Euclidean billiards: it is, roughly, $rd$.
Since we count $\Z_r$-orbits of periodic Finsler billiard trajectories, rather than $D_r$-orbits, one might expect the numbers $N_F\big(M^{d-1},r\big)$ to be about twice as large as the numbers $N_E\big(M^{d-1},r\big)$.

For example, consider Euclidean billiard inside a strictly convex closed smooth hypersurface, and switch on a weak magnetic field. One may expect each periodic trajectory in absence of the magnetic field to give rise to two periodic magnetic billiard trajectories, see Fig.~\ref{triangle}.

\begin{figure}[ht!]\centering
\includegraphics[width=1.4in]{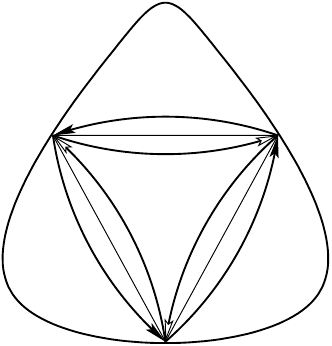}

\caption{A $3$-periodic billiard trajectory giving rise to two $3$-periodic magnetic billiard trajectories.}
\label{triangle}
\end{figure}

In dimension two, this is indeed the case: the Finsler billiard map is an area preserving twist map, and it has two $r$-periodic trajectories for every rotation number $k$ coprime with $r$. If the metric is symmetric, then the orbits corresponding to the rotation numbers $k$ and $r-k$ differ only by the orientation, and they are counted as one, but in the asymmetric case, these are indeed different orbits.

We do not know how close the lower bounds of Theorem~\ref{mainthm} are to being sharp. One may expect a notable difference between the reversible and non-reversible cases. In the related problem of closed geodesics, it is known that every Riemannian metric on the 2-sphere possesses infinitely many geometrically distinct closed geodesics \cite{Bangert, Franks}, but a Finsler metric may have precisely two distinct prime closed geodesics, as in the well known Katok example~\cite{Katok} (in which the two closed geodesics are the inverses of each other). Do similar examples exist for Finsler billiards?
\end{Remark}

\section[From geometry of billiards to topology of cyclic configuration spaces]{From geometry of billiards to topology\\ of cyclic configuration spaces}\label{tech}

\subsection{Introduction to Finsler geometry}\label{Fgeo}

We begin with a very brief introduction to Finsler geometry. For a thorough treatment, see \cite{Alvarez-Duran,Bao,Shen}.

A Finsler metric on a smooth manifold $U$ is determined by a smooth non-negative fiberwise-convex Lagrangian function $L\colon  TU \longrightarrow [0,\infty)$, with the property that on each tangent spa\-ce~$T_yU$, $L$~is positively homogeneous: $L(y,tv) = tL(y,v)$ for non-negative $t$ and positive off the zero section. The restriction of $L$ to any tangent space $T_yU$ gives the Finsler length of vectors in~$T_yU$.

The vectors in $T_yU$ of unit Finsler length form a strictly convex hypersurface $I \subset T_yU$, called the \emph{indicatrix}, which plays the role of the unit sphere in Riemannian geometry. We make the additional assumption that each indicatrix is quadratically convex. Specifying a smooth field of indicatrices on $U$ is an equivalent method of defining a Finsler metric on~$U$.

The Finsler metric on $TU$ induces a notion of distance on the base manifold~$U$. The length of a smooth curve $\gamma\colon  [a,b] \longrightarrow U$ is given by the integral
\[
\operatorname{Length}(\gamma) = \int_a^b L(\gamma(t),\gamma'(t))\,{\rm d} t.
\]
The length of $\gamma$ is independent of the parametrization, and a Finsler geodesic is an extremal of the length functional. In particular, a Finsler geodesic $\gamma$ satisfies the Euler--Lagrange equations:
\[
L_{vv}\big(\gamma(t),\gamma'(t)\big) \cdot \gamma^{\prime\prime}(t) + L_{vy}\big(\gamma(t),\gamma'(t)\big) \cdot \gamma^{\prime}(t) = L_y\big(\gamma(t),\gamma'(t)\big).
\]
In this formula we use a shorthand notation, so that $L_v$ is the vector $(L_{v_1},\ldots,L_{v_n})$ and $L_{vv}$ is the matrix $(L_{v_i v_j})$,
etc.

For each pair of points $x$ and $y$ there is a corresponding \emph{Finsler distance} $f(x,y)$, equal to the length of the shortest oriented geodesic from $x$ to $y$. We stress that since geodesics are not necessarily reversible, this distance function $f$ need not be symmetric, and therefore $f$ is not a~genuine metric on~$U$.

The \emph{figuratrix} $J \subset T_y^*U$ is the ``unit sphere of the cotangent space'', defined as follows: for a vector $v$ in the indicatrix $I \subset T_yU$, there is a unique covector $D_v$, defined by the properties that $\operatorname{Ker}(D_v) = T_vI$ and $D_v(v) = 1$. The map $D\colon I \longrightarrow T_y^*U$ which maps $v$ to $D_v$ is the \emph{Legendre transform}, and the image is the figuratrix $J$. The dual transform $D^*\colon J \longrightarrow I$ is defined similarly. The Legendre transform is an involution: the composition of $D$ and $D^*$ is the identity map.
Rightfully, the Legendre transform is a smooth bundle map from the indicatrix bundle to the figuratrix bundle, but we will often use the notation $D\colon I \longrightarrow J$ when the basepoint $y$ is understood.

\subsection{The Finsler billiard map}\label{sec:reflectionlaw}

Let $U$ be a smooth $d$-dimensional Finsler manifold. Let $X \subset U$ be a compact $d$-dimensional submanifold with boundary $M = \partial X$. We assume that $X$ and $M$ satisfy the strict convexity assumptions formulated in Section~\ref{BilFin}. We will refer to~$X$ as a \emph{billiard table}.

Let $xy$ and $yz$ be two oriented geodesic segments, where $x,z \in X$ are in the interior of the billiard table, and $y \in M$ is on the boundary. We say that $yz$ is the \emph{Finsler billiard reflection} of~$xy$ if $y$ is a critical point of the distance function $f(x,\cdot) + f(\cdot,z)\colon M \longrightarrow [0,\infty)$. This relation is \emph{not symmetric}, so it does not imply that $yx$ is the billiard reflection of~$zy$.

We describe the reflection law in the context of the Finsler setup following the treatment in~\cite{Gutkin2002}. Although the Finsler metric there is assumed to be symmetric, the reflection law is the same.

For each $y \in M$, let $p \in J \subset T_y^*U$ be the \emph{conormal}, defined as the unit cotangent vector which vanishes on $T_yM$ and is positive on the outward vectors. Let~$xy$ be an incoming geodesic and $yz$ the reflected outgoing geodesic, corresponding respectively to tangent vectors~$u$ and~$v$ in~$I$. The reflection $xy$ to~$yz$ manifests as the following relation in the cotangent space.

\begin{Lemma}[Finsler billiard reflection law]\label{lem:cotangentlaw}
The covector $D_u - D_v$ is conormal to $T_yM$; in particular $D_u - D_v = tp$ for some $t > 0$, see Fig.~{\rm \ref{reflfig}}.
\end{Lemma}

\begin{figure}[ht!]\centering
\includegraphics[width=3.4in]{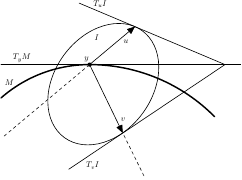}
\caption{The Finsler billiard reflection in dimension two.}\label{reflfig}
\end{figure}

\begin{proof}The point $y \in M$ is a relative extremum of the function $f(x,\cdot) + f(\cdot,z)$, so the differential of this function is conormal to the tangent hyperplane $T_yM$. Let us compute this differential to show that it is $D_u - D_v$.

Fix a point $x \in X$, and consider the wave propagation from $x$. Let $c$ be a non-singular point of the wave front $\mathcal{F}_{t_0}$ such that the oriented geodesic segment $xc$ of length $t_0$ is contained in the interior of $X$. The Finsler length function $f(x,\cdot)$ extends to a smooth function in a neighborhood of~$c$. More precisely, for every point $d$ sufficiently close to $c$, there exists a $t$ near $t_0$ such that $d \in \mathcal{F}_t$. Let $I$ and $J$ represent the indicatrix and figuratrix at~$c$, and let $u \in I$ be the Finsler unit vector along~$xc$.

We claim that, at the point $c$, ${\rm d}f(x,\cdot) = D_u$. The wave front $\mathcal{F}_{t_0}$ is a level set of $f(x,\cdot)$. Hence ${\rm d}f(x,\cdot)$ annihilates the tangent space to $\mathcal{F}_{t_0}$. By the Huygens principle, $D_u$ is conormal to this plane, hence proportional to ${\rm d}f(x,\cdot)$. It also follows from the Huygens principle that ${\rm d}f(x,\cdot)(u) = 1$, which is true of $D_u$ by definition. Hence ${\rm d}f(x,\cdot) = D_u$.

\looseness=-1 Since the Finsler distance function is non-symmetric, we also need to understand ${\rm d}f(\cdot,z)$. Consider the oriented geodesic segment $yz$. Although $zy$ is not necessarily a geodesic with respect to the Finsler structure $L$, it is a geodesic with respect to the ``reverse'' Finsler structure on $M$ defined by the Lagrangian $\bar{L}(v) \coloneqq L(-v)$. For the corresponding Finsler distance function~$\bar{f}$ we have the equality $\bar{f}(z,y) = f(y,z)$, and for the corresponding Legendre transform~$\bar{D}$ we have the equality $\bar{D}_{-v} = -D_v$ (by definition of the covector $D_v$). Therefore we have ${\rm d}f(\cdot,z) = {\rm d}\bar{f}(z,\cdot) = \bar{D}_{-v} = -D_v$, where the middle equality follows from the paragraph above. This completes the proof.
\end{proof}

\subsection{Magnetic billiards as Finsler billiards}\label{magnetic}

A popular billiard model that has been extensively studied in the last decades are magnetic billiards \cite{Berglund,Bialy,GutkinB,Robnik,Tabachnikov2004,Tasnadi,Zharnitsky}.

On a Riemannian surface, a magnetic field is given by a function $B(x)$, and the motion of a~charged particle is described by the differential equation $\ddot x = B(x) J \dot x$, where $J$ is the rotation of the tangent plane by $\pi/2$. This equation implies that the speed of the particle remains constant (the Lorentz force is perpendicular to the direction of motion). In the Euclidean plane, if the magnetic field is constant, the trajectories are circles of a fixed radius (Larmor circles).

In general, a magnetic field on a Riemannian manifold~$M$ is a closed differential 2-form $\beta$, and the magnetic flow is the Hamiltonian flow of the usual Hamiltonian $|p|^2/2$ on the cotangent bundle $T^* M$ with respect to the twisted symplectic structure $\omega + \pi^*(\beta)$, where $\omega={\rm d}p\wedge {\rm d}q$ is the standard symplectic form and $\pi\colon T^*M \longrightarrow M$ is the projection.

Magnetic billiard describes the motion of a charged particle confined to a domain with elastically reflecting boundary. The reflection law is the same as for the usual billiards, with zero magnetic field: the angle of incidence equals the angle of reflection. Following~\cite{Tabachnikov2004}, one can interpret magnetic billiards as Finsler ones.

Let us assume that the magnetic 2-form is exact: $\beta={\rm d}\alpha$ for some differential 1-form on a~Riemannian manifold $M^n$. Then the magnetic flow admits a Lagrangian formulation with the Lagrangian function{\samepage
\[
\bar L(x,v) = \tfrac{1}{2} |v|^2 + \alpha(x)(v),
\]
where $x \in M$, $v \in T_x M$, and $|v|$ is the Riemannian norm of the tangent vector~$v$.}

We want to consider the motion with unit speed, that is, to fix an energy level. Following the Maupertuis principle, replace the Lagrangian $\bar L$ with
\begin{equation} \label{magLagr}
L(x,v) = |v| + \alpha(x)(v).
\end{equation}
Assume that the magnetic field is weak enough so that $L(x,v) >0$ for all non-zero tangent vectors $v$ and all points $x$, that is, we assume that $|\alpha(x)| < 1$ everywhere. Then formula (\ref{magLagr}) defines a non-symmetric Finsler metric.

\begin{Lemma}[magnetic billiard reflection law]\label{magrefllaw}
The indicatrix of the magnetic Finsler metric~\eqref{magLagr} is an ellipsoid of revolution with a focus at the origin, and the Finsler billiard reflection law coincides with the usual one: the angle of incidence equals the angle of reflection.
\end{Lemma}

\begin{proof}Fix a point $x$ and consider the tangent space at this point. The tangent space is Euclidean, and the indicatrix $I$ is given by the equation $|v| + \alpha (v)=1$. Let $e$ be the tangent vector dual to the covector $\alpha$, that is, $\alpha (v) = e\cdot v$. We can choose an orthonormal basis so that $e=(t,0,\ldots,0)$ with $t<1$.

The equation of the indicatrix is $L(x,v)=1$, or $x_1^2+\dots + x_n^2 = (1-tx_1)^2$, that is,
\[
\big(1-t^2\big)^2 \left(x_1 + \frac{t}{1-t^2} \right)^2 + \big(1-t^2\big) \big(x_2^2+\dots+x_n^2\big) = 1.
\]
This is the equation of an ellipsoid of revolution obtain by revolving the ellipse
\[
\frac{(x_1+c)^2}{a^2} + \frac{x_2^2}{b^2} =1,
\]
where
\[
a^2=\frac{1}{(1-t^2)^2},\qquad b^2 = \frac{1}{(1-t^2)},\qquad c=\frac{t}{1-t^2},
\]
and hence $c^2 = a^2 - b^2$. Therefore this ellipse has a focus at the origin.

\begin{figure}[ht!]\centering
\includegraphics[width=2.3in]{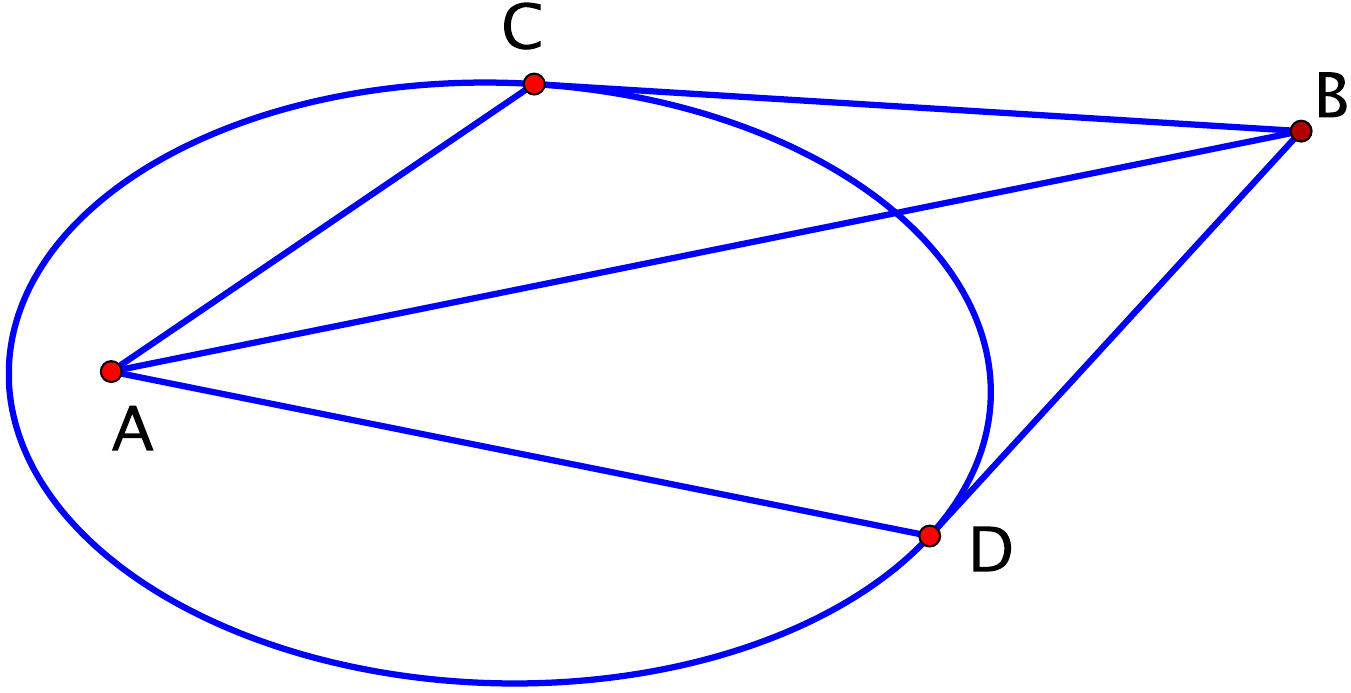}
\caption{Point $A$ is a focus of an ellipse; then $\angle BAC = \angle BAD$.}\label{ellipse}
\end{figure}

The second statement of the lemma reduces to a known geometric property of conics depicted in Fig.~\ref{ellipse}, see, e.g., \cite{Akopyan}. Compare with the Finsler reflection law, Fig.~\ref{reflfig}.
\end{proof}

We assume that the convexity conditions of Section~\ref{BilFin} hold for magnetic billiards. This implies that the magnetic field is weak enough. For example, if the billiard table is a planar domain bounded by a smooth strictly convex curve and the magnetic field is constant, then the Larmor radius is greater than the greatest radius of curvature of the boundary curve. This is one of the three regimes in magnetic billiards described in~\cite{Robnik}, in which ``motion is qualitatively similar to the field-free case'', albeit not time-reversible.

\subsection{Morse-theoretic approach to Finsler billiards}
\label{sec:morse}

We now discuss the necessary Finsler geometry and topology to apply results from Morse and Lusternik--Schnirelmann theories to the problem of periodic Finsler billiard trajectories. Similar preparation work is easier to do in the Euclidean case; it was the content of \cite[Section~4]{Farber-Tabachnikov2002-01}.

The \emph{ordered cyclic configuration space} of $r$ consecutively distinct points on $M$ is the space
\[
\GG(M,r)\coloneqq \big\{(x_1,\ldots,x_r)\in M^{\times r}\colon x_i\neq x_{i+1}\ \text{for all}\ i\big\},
\]
where by convention $x_{r+1}=x_1$. As mentioned in the Introduction, we consider the \emph{length function}
\[
\Lambda \colon \ \GG(M,r) \longrightarrow \R, \qquad \Lambda(x_1,\dots,x_r) = f(x_1,x_2) + f(x_2,x_3) + \cdots + f(x_r,x_1).
\]
The function $\Lambda$ is smooth and $\Z_r$-equivariant, but in contrast with Euclidean billiards, $\Lambda$ need not be $D_{r}$-equivariant. By definition of the Finsler billiard reflection for geodesic rays (Section~\ref{sec:reflectionlaw}), the $r$-periodic Finsler billiard orbits are precisely the critical points of $\Lambda$.

We would like to use this fact by applying Morse and Lusternik--Schnirelmann theories to obtain a lower bound for the number of periodic Finsler billiard orbits; however, these theories cannot be applied directly because $\GG(M,r)$ is not a closed manifold. We will show that for sufficiently small $\varepsilon$ we can replace $\GG(M,r)$, without affecting the topology, by the following compact manifold with boundary:
\[
\GG_\varepsilon(M,r) := \bigg\{ (x_1,\dots,x_r) \in M^{\times r} \, \bigg| \, \prod_{i=1}^r f(x_i,x_{i+1}) \geq \varepsilon > 0 \bigg\},
\]
where again the indices are understood cyclically. Similar to the Euclidean case \cite[Proposition~4.1]{Farber-Tabachnikov2002-01}, we establish the following proposition in the Finsler setup.

\begin{Proposition}
\label{prop:main}
If $\varepsilon > 0$ is sufficiently small, then
\begin{enumerate}\itemsep=0pt
\item[\rm (1)] $\GG_\varepsilon(M,r)$ is a smooth manifold with boundary;
\item[\rm (2)] the inclusion $\GG_\varepsilon(M,r) \subset \GG(M,r)$ is a $\Z_r$-equivariant homotopy equivalence;
\item[\rm (3)] all critical points of $\Lambda\colon\GG(M,r) \longrightarrow \R$ are contained in the interior of $\GG_\varepsilon(M,r)$;
\item[\rm (4)] at every critical point of $\Lambda \big|_{\partial \GG_\varepsilon(M,r)}$, the differential ${\rm d}\Lambda$ is positive on inward vectors.
\end{enumerate}
\end{Proposition}

This proposition makes it possible to apply $\Z_r$-equivariant Morse and Lusternik--{S}chnirel\-mann theories to the length function $\Lambda$ on the manifold with boundary $\GG_\varepsilon(M,r)$. We shall restrict ourselves to the case when $r$ is prime. Then the group $\Z_r$ acts freely on $\GG_\varepsilon(M,r)$, and $\Z_r$-equivariant Morse theory reduces to Morse theory on the quotient manifold $\GG_\varepsilon(M,r)/\Z_r$.

Due to item (4) of Proposition~\ref{prop:main}, the topological lower bounds on the number of critical points of $\Lambda$, such as the sum of Betti numbers or the Lusternik--{S}chnirelmann category, come from the topology of $\GG_\varepsilon(M,r)/\Z_r$. This space has the same topology as the cyclic configuration space $\GG\big(S^{d-1},r\big)/\Z_r$, due to item (2), and the fact that~$M$ is topologically the sphere. We refer to~\cite{Laudenbach} for the Morse theory on manifolds with boundary.

\subsection{Technical bounds for Finsler billiards}
\label{techbounds}
We start with a number of technical lemmas working toward the proof of Proposition \ref{prop:main}.

The boundary $\partial \GG_\varepsilon(M,r)$ is the level set (at $\varepsilon^2$) of the smooth function
\[
F \colon \ M^{\times r} \longrightarrow \R,  \qquad F(x_1,\dots,x_r) = \prod_{i=1}^r f(x_i,x_{i+1})^2,
\]
and $F^{-1}(0) = M^{\times r} \setminus \GG(M,r)$ is a critical level. The first two items are therefore a consequence of the following lemma.

\begin{Lemma}
\label{lem:delta}
There exists a constant $\delta > 0$ such that the interval $(0,\delta)$ consists of regular values of $F$.
\end{Lemma}

We offer some geometric intuition for this statement. As indicated above, the critical points of the length function $\Lambda$ are precisely the periodic Finsler billiard orbits. Similarly, we may think of the critical points of the function $F$ as the periodic orbits of some ``unusual'' billiard trajectory, which we will call the $F$-\emph{billiard} trajectory, and for which the reflection law is given by Lemma \ref{lem:reflectionlaw2} below. In this terminology, Lemma \ref{lem:delta} claims the existence of $\delta$ such that any $r$-periodic $F$-billiard orbit $(x_1,\dots,x_r)$ satisfies $F(x_1,\dots,x_r) \geq \delta$.

Assume that no such $\delta$ exists. Then we can find a closed $F$-billiard orbit such that one of the edge lengths $\ell_i \coloneqq f(x_{i-1},x_i)$ is arbitrarily small. The contradiction will arise in Section~\ref{sec:lemdeltaproof}, where we show the following two statements:
\begin{enumerate}\itemsep=0pt
\item[\rm (1)] If one edge of a closed $F$-orbit is ``arbitrarily short'', then \emph{all} of the edges are ``arbitrarily short''.
\item[\rm (2)] A closed $F$-orbit cannot have all edges ``arbitrarily short''.
\end{enumerate}

We will not explicitly define \emph{arbitrarily short}, but a suitable quantity could be determined in terms of the period $r$ and certain ``curvature'' quantities, which depend on the Finsler geo\-metry of~$M$ and on the geometry of the indicatrix bundle. A similar shortness statement was proven in the Euclidean case \cite{Farber-Tabachnikov2002-01}; however, those curvature estimates rely on symmetry of the inner product, symmetry of the unit tangent spheres, and also some trigonometry; none of which we have at our disposal in the Finsler setting. We treat these subtleties in the following Sections~\ref{sec:curvature}--\ref{sec:indicatrixbound}.

\subsubsection{A ``curvature'' bound for Finsler manifolds}\label{sec:curvature}

In the Finsler setting, we would like to formalize the intuitive idea that for $y \in M$, ``a geodesic segment $yz$ is short if and only if the corresponding tangent vector at $y$ is almost tangent to~$M$''.

This is easy to establish in the Euclidean case \cite{Farber-Tabachnikov2002-01} as follows. The boundary $M$ of the billiard table is a strictly convex hypersurface in Euclidean space~$\R^d$. By strict convexity of $M$, there exist positive numbers $\rho<R$ such that for every $y \in M$, there exist two spheres tangent to $M$ at $y$, of radii $\rho$ and $R$, such that the sphere of radius $R$ contains $M$, and the sphere of radius $\rho$ is contained in $M$. Let $n$ be the outward unit normal at $y$ and let $v \in T_y\R^d$ be a unit tangent vector with $\langle v, n \rangle < 0$, and follow the geodesic from $y$ in the direction of $v$ until colliding again with the boundary, say at $z \in M$. Then the numbers $\rho$ and $R$ satisfy
\[
\rho < \frac{|y-z|}{-2\langle v, n \rangle} < R.
\]
In particular, \emph{the measurements} $g(v) \coloneqq |y-z|$ \emph{and} $-p(v) = -\langle v, n \rangle$ \emph{are of the same order}. We will use the notation $g \sim p$ for such a statement.

Now let $U$ be a Finsler manifold and $M$ a smooth, closed hypersurface, quadratically convex with respect to the Finsler geodesics. For $y \in M$, let $p \in J$ represent the unit covector which is conormal to $M$ at $y$ and positive on outward vectors. Given $v \in I \subset T_yU$ such that $p(v) < 0$, let $z \in M$ be the first collision with $M$ of the geodesic ray emanating from $y$ in the direction $v$. Define $g(v) = f(y,z)$.

Similarly, given $u \in I \subset T_yU$ such that $p(u) > 0$, let $x \in M$ be the point such that the oriented geodesic ray $xy$ enters $y$ with direction $u$. Define $g(u) = f(x,y)$.

\begin{Remark}To verify that such a point $x$ exists, note that the oriented ray $xy$ is a geodesic if and only if the oriented ray $yx$ is a geodesic with respect to the ``reverse'' Finsler structure defined by the Lagrangian $\bar{L}(u) \coloneqq L(-u)$. The point $x$ can be defined as the first collision with~$M$ of the reverse geodesic ray emanating from~$y$ in the direction $-u$.
\end{Remark}

\begin{Lemma}
\label{lem:finslerbound}
There exist positive constants $\rho < R$ such that, for all $y \in M$, all $v \in I \subset T_yU$ with $p(v) < 0$, and all $u \in I \subset T_yU$ with $p(u) > 0$, we have
\[
\rho < \frac{f(y,z)}{-p(v)} < R \ \ \ \ \ \mbox{ and } \ \ \ \ \ \rho < \frac{f(x,y)}{p(u)} < R.
\]
Equivalently, the functions $g$ and $p$ have the same order, $g \sim p$.
\end{Lemma}

The following proof is due to Sergei Ivanov.

\begin{proof}
We will show the left inequalities, corresponding to vectors $v$ emanating from $y$. The right inequalities follow similarly.

By compactness of $M$, it suffices to show existence for a single point $y \in M$. In fact, it is enough to show that the lemma holds near $y$: that is, for $v \in I \subset T_yU$ with $p(v)$ near $0$. This is due to the following fact: for all $v$ such that $p(v)$ is sufficiently away from $0$, the quantity $\frac{f(y,z)}{-p(v)}$ is bounded away from zero. Indeed, as $f(y,z)$ tends to zero, that is, point $z$ tends to point $y$, the vector along the geodesic $yz$ tends to a tangent vector to $M$ at $y$, that is, $p(v)$ also tends to zero.

Consider smooth coordinates $(y_1,\dots,y_d)$ near $y$, such that a neighborhood of $y$ in $M$ is given by the equation $y_d = 0$. Let $\gamma(t) = (\gamma_1(t), \dots, \gamma_d(t))$ be the unit-speed geodesic passing through~$y$ at time $0$ and with unit tangent vector $v \in I$ (see Fig.~\ref{fig:gamma}, left).

\begin{figure}[ht!]\centering
\includegraphics[width=2.8in]{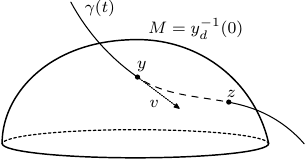} \qquad \includegraphics[width=2.8in]{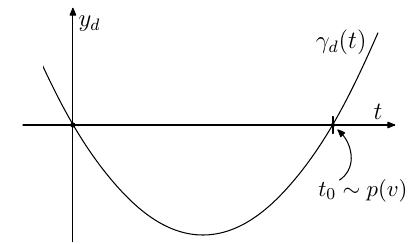}
\caption{(Left): The unit speed geodesic $\gamma(t) = (\gamma_1(t),\dots,\gamma_d(t))$ passes through $y \in M$ at time $0$ with tangent vector $v$ and again at $z \in M$ at time $t_0 \sim p(v)$. (Right): The $d^{\mbox{\scriptsize th}}$ component of $\gamma$, shown to the second order.}\label{fig:gamma}
\end{figure}

In case $v \in I \cap T_yM$, we have $\gamma_d(0) = 0$ and $\gamma_d^\prime(0) = 0$, so by quadratic convexity, $\gamma_d^{\prime\prime}(0)$ is bounded away from $0$ for all $v \in I \cap T_yM$. Therefore, for $v$ almost tangent to $M$ (i.e., with $p(v) < 0$ sufficiently near $0$), we have $\gamma_d^{\prime\prime}(0) \sim 1$. In addition, for $v$ almost tangent to $M$, we have $p(v) \sim \gamma_d^\prime(0) < 0$.

Now writing $\gamma_d(t)$ up to second order yields
\[
\gamma_d(t) = \gamma_d(0) + \gamma_d^\prime(0)t + \tfrac12 \gamma_d^{\prime\prime}(0)t^2,
\]
as depicted in Fig.~\ref{fig:gamma}, right. Thus $\gamma_d(t)$ will meet zero again at a time
\[
t_0 = -2 \frac{\gamma_d^\prime(0)}{\gamma_d^{\prime\prime}(0)} \sim p(v).
\]
Since $t_0$ is the length $f(y,z) = g(v)$, we conclude that $g \sim p$, completing the proof.
\end{proof}

\subsubsection[Reflection law for the $F$-billiard]{Reflection law for the $\boldsymbol{F}$-billiard}
\label{sec:reflectionF}
Just as the critical points of the length function $\Lambda$ correspond to periodic orbits of the Finsler billiard, we think of the critical points of~$F$ as the periodic orbits of some ``unusual'' billiard trajectory determined by the function~$F$. In fact, what we are really interested in are the critical points of~$F$; we use this billiard terminology because it is intuitive and suggestive.

So suppose that $x,y,z \in M$ are points such that the oriented geodesic segment $yz$ is the $F$-reflection of the oriented geodesic segment $xy$. Let $u, v \in I \subset T_yU$ be the tangent vectors which correspond, respectively, to $xy$ and~$yz$. Let $p \in J \subset T^*_yU$ represent the outward-pointing unit conormal, and let $n = D^*_p$ represent the outward pointing unit normal. Then the $F$-billiard has the following reflection law.

\begin{Lemma}\label{lem:reflectionlaw2}
The $F$-billiard reflection law is given by the following cotangent relation:
\[
\frac{D_u}{f(x,y)} - \frac{D_v}{f(y,z)} = tp, \qquad t>0.
\]
\end{Lemma}

\begin{proof}
The point $y \in M$ is a relative extremum of the function $F(\cdot) \coloneqq f(x,\cdot)^2f(\cdot,z)^2$, so the differential of this function at $y$ is conormal to the tangent hyperplane $T_yM$. Therefore the differential of the function
\[
\tfrac{1}{2}\ln F(\cdot) = \ln f(x,\cdot) + \ln f(\cdot, z)
\]
at $y$ is also conormal to the tangent hyperplane $T_yM$. We compute
\[
\tfrac{1}{2}{\rm d} \big(\ln F(\cdot)\big) = {\rm d}\big(\ln f(x,\cdot) + \ln f(\cdot,z)\big) = \frac{{\rm d}f(x,\cdot)}{f(x,\cdot)} + \frac{{\rm d}f(\cdot,z)}{f(\cdot,z)}.
\]
Recall from the proof of Lemma \ref{lem:cotangentlaw} that at the point $y$, ${\rm d}f(x,\cdot) = D_u$ and ${\rm d}f(\cdot,z) = -D_v$. Therefore, at the point $y$, we have
\[
\frac{D_u}{f(x,y)} - \frac{D_v}{f(y,z)} = tp,
\]
as desired.
\end{proof}

We will also make use of the following consequence.

\begin{Lemma}\label{lem:tangentlaw}If $v \in I \subset T_yM$ is the $F$-reflection of $u \in I \subset T_yM$, then either the linear subspaces $T_uI, T_vI, T_yM \subset T_yU$ are equal, or they intersect in a subspace of codimension $2$.
\end{Lemma}

\begin{proof}
First of all, we clarify the statement. The indicatrix $I$ is a hypersurface in the vector space $T_yU$, hence its tangent spaces can be considered as vector subspaces of $T_yU$. It is in this sense that we intersect them, and $T_yM$, in the ambient space $T_yU$. This remark applies to other similar arguments elsewhere in the paper.

Now, by the reflection law, the covectors $D_u$, $D_v$, and $p$ are linearly dependent, hence span a~subspace of at most two dimensions. Therefore, the intersection of their kernels, $T_uI \cap T_vI \cap T_yM$, has codimension at most two.
\end{proof}

Given $u \in I \subset T_yU$ with $p(u) > 0$, let $v$ be the $F$-reflection of $u$. Suppose that $u \neq n$, so that the hyperplanes $T_uI$, $T_vI$, and $T_yM$ are not parallel. Then, from Lemma \ref{lem:tangentlaw}, the intersection of the linear hyperplanes $T_uI \cap T_vI$ is a codimension $1$ subspace of $T_yM$, hence there is a unique vector $w = w(u) \in T_yM \cap I$ for which the kernel of~$D_w$ contains $T_uI \cap T_vI$ and such that $D_u(w) > 0$. Now applying the (cotangent) relation of Lemma~\ref{lem:reflectionlaw2} to the vector~$w$ yields:
\begin{gather}\label{eqn:tangentreflection}
\frac{D_u(w)}{f(x,y)} = \frac{D_v(w)}{f(y,z)}.
\end{gather}

Note that $D_v(w) > 0$, since the denominators are positive and $D_u(w) > 0$ by assumption.

%%%
\subsubsection{Global indicatrix bound}
\label{sec:indicatrixbound}

Since the Finsler norm need not arise from an inner product, it is useful to develop some method of comparing two vectors $u, v \in I$. One non-symmetric idea is to apply the Legendre transform to obtain a covector $D_u$ which then acts on $v$. We have a general lower bound by compactness and a specific upper bound by strict convexity of the indicatrices.

\begin{Lemma}
\label{lem:upperbound}
There exists a positive constant $K \geq 1$ such that, for every triple $y \in M$, $u,v \in I \subset T_yU$, one has $-K \leq D_u(v) \leq 1$. In particular, $D_u(v) = 1$ if and only if $u = v$.
\end{Lemma}

We will also require a more refined bound. For $y \in M$, consider $u \in I \subset T_yU$ such that $p(u) > 0$ and $u \neq n$. Let $v \in I \subset T_yU$ be the $F$-reflection of $u$ and let $w$ be as described in Section \ref{sec:reflectionF}. Although $w = w(u)$ is not defined when $u = n$, any sequence $u_i \longrightarrow n$ has $D_{u_i}(w_i) \longrightarrow 0$ and $D_{v_i}(w_i) \longrightarrow 0$, so that the maps $u \longmapsto D_u(w)$ and $u \longmapsto D_v(w)$ extend continuously to $u = n$.

\begin{Lemma}\label{lem:lowerbound}There exists a positive constant $k<1$ such that for all $y \in M$ and all $u \in I \subset T_yU$ with $p(u) > 0$,
\[
k < D_u(w) + p(u) \qquad \text{and} \qquad k < D_v(w) - p(v).
\]
\end{Lemma}

\begin{proof}The numbers $D_u(w)$, $D_v(w)$, $p(u)$, and $-p(v)$ are all positive. The first two are bounded away from $0$ except when $u$ is almost parallel to~$n$, and the last two are bounded away from~$0$ except when $u$ is almost tangent to~$M$.
\end{proof}

\subsection[The $F$-billiard reflection preserves shortness and almost-tangency]{The $\boldsymbol{F}$-billiard reflection preserves shortness and almost-tangency}

With all of the technical bounds obtained in Section~\ref{sec:morse}, we are ready to study the $F$-billiard reflection in more detail. We show that short geodesics $F$-reflect to short geodesics, and that almost-tangent vectors $F$-reflect to almost-tangent vectors.

In particular, suppose that $u \in I \subset T_yU$, let $v \in I \subset T_yU$ be the $F$-reflection of $u$, and let~$x$ and~$z$ be points in~$M$ such that $u$ and $v$ correspond, respectively, to the oriented geodesic segments $xy$ and $yz$. In this sense, we may consider $v$, $x$, and $z$ as functions of $u$. Let $p \in J \subset T_y^*U$ be the outward pointing unit conormal and let $n = D^*_{p}$ be the outward pointing unit normal. Our goal is to show the following:
\[
p(u) \sim f(x,y) \sim f(y,z) \sim p(v).
\]
Note that the outer equivalences were already established in Lemma \ref{lem:finslerbound}.

We make use of the positive constants $\rho$ and $R$ defined by Lemma \ref{lem:finslerbound}, as well as the positive constants $K$ and $k$ determined in Lemmas \ref{lem:upperbound} and \ref{lem:lowerbound}.

\begin{Lemma}
\label{lem:lengthequiv}
If $v \in I \subset T_yU$ is the $F$-reflection of $u \in I \subset T_yU$, with $p(u) > 0$, then
\[
\frac{k\rho}{K(\rho+R)} \leq \frac{f(x,y)}{f(y,z)} \leq \frac{K(\rho+R)}{k\rho}.
\]
In particular, the measurements $f(x,y)$ and $f(y,z)$ are equivalent: $f(x,y) \sim f(y,z)$.
\end{Lemma}

\begin{proof}
We prove by contradiction. Assume that the left inequality fails, so that
\begin{align}
\label{eqn:leftfail}
\frac{f(x,y)}{f(y,z)} < \frac{k\rho}{K(\rho+R)}.
\end{align}
Then the $F$-reflection law applied to $w$ (\ref{eqn:tangentreflection}), combined with (\ref{eqn:leftfail}) and then Lemma \ref{lem:upperbound} twice, yields
\begin{align}
\label{eqn:estimate1}
D_u(w) = \frac{f(x,y)}{f(y,z)}D_v(w) < \frac{k\rho}{K(\rho+R)}D_v(w) \leq \frac{k\rho}{K(\rho+R)} \leq \frac{k\rho}{\rho+R}.
\end{align}
On the other hand, Lemma \ref{lem:finslerbound}, combined with (\ref{eqn:leftfail}) and then Lemma \ref{lem:upperbound}, yields
\begin{align}
\label{eqn:estimate2}
p(u) < - \frac{kR}{K(\rho+R)} p(v) < \frac{kR}{\rho+R}.
\end{align}

But the sum of (\ref{eqn:estimate1}) and (\ref{eqn:estimate2}) gives a contradiction; the right side is equal to $k$, while the left side is greater than $k$ by Lemma \ref{lem:lowerbound}.

The right inequality of Lemma \ref{lem:lengthequiv} can be shown similarly.
\end{proof}

\begin{Lemma}
\label{lem:context}
The $F$-billiard reflection $u \longmapsto v$ can be extended continuously to $I \cap T_yM$. Moreover, this extension is the identity map on $I \cap T_yM$.
\end{Lemma}

\begin{proof}
This lemma clearly holds for the ordinary Finsler billiard reflection law in Lemma \ref{lem:cotangentlaw}. In the $F$-billiard case, by Lemmas \ref{lem:finslerbound} and \ref{lem:lengthequiv}, we have $p(u) \sim f(x,y) \sim f(y,z) \sim p(v)$, so if a~sequence $u_n \subset I$ tends to $T_yM$, so does the sequence of $F$-reflections $v_n$.

From Lemma~\ref{lem:tangentlaw}, we have $\operatorname{Ker} D_{u_n} \cap T_yM = \operatorname{Ker} D_{v_n} \cap T_yM$ for every $n$. Therefore, if $u_n \longrightarrow u \in I \cap T_yM$, then $\lim \big(\operatorname{Ker} D_{u_n} \big) \cap T_yM$ is some codimension~1 subspace $W \subset T_yM$, and we must have $W = \operatorname{Ker} D_v \cap T_yM$ for any continuously-defined $F$-reflection $v$ of $u$. Since we also must have $v \in I \cap T_yM$, there are only two candidates for $v$: $u$ itself and the ``opposite'' vector~$u'$, for which $D_u(u') < 0$. Now, by definition of the unit vector $w_n \in T_yM$ (see Section~\ref{sec:reflectionF}), $u = \lim w_n$. We have $D_{v_n}(w_n) > 0$ for all $n$, so $D_v(u)$ must be nonnegative. Hence $v = u$.
\end{proof}

We are now ready to prove Lemma \ref{lem:delta}.

\subsection[Nonexistence of short periodic $F$-billiard trajectories]{Nonexistence of short periodic $\boldsymbol{F}$-billiard trajectories} \label{sec:lemdeltaproof}

\begin{proof}[Proof of Lemma \ref{lem:delta}]
Assume that $(x_1,\dots,x_r)$ is an $r$-periodic $F$-billiard orbit. Let $u_i, v_i \in I \subset T_{x_i}U$ represent the tangent vectors which correspond, respectively, to the oriented geodesic segments $x_{i-1}x_i$ and $x_ix_{i+1}$. Let $p_i \in J \subset T_{x_i}^*U$ be the outward pointing unit conormal and let $n_i = D^*_{p_i}$ be the outward pointing unit normal. Let $\ell_i = f(x_{i-1},x_i)$. We aim to show that no $\ell_i$ can be arbitrarily small. From Lemma~\ref{lem:lengthequiv} we obtain the edge comparison for any~$i$ and~$j$:
\[
\frac{\ell_i}{\ell_j} < \left( \frac{K(\rho+R)}{k\rho} \right)^r.
\]
In particular, this establishes that $\ell_i \sim \ell_j$, so that if $\ell_i$ is small for any $i$, every~$\ell_j$ is also small.

Now, seeking a contradiction, we assume that a critical $r$-gon $(x_1,\dots,x_r)$ has all edges arbitrarily short. Then each~$u_i$ is almost tangent, so by Lemma~\ref{lem:context}, $u_i \sim v_i$ in any auxiliary metric on~$M$. In particular, for any auxiliary Euclidean metric on $M$, $(x_1,\dots,x_r)$ is an $r$-gon with all exterior angles small. The contradiction will arise once we show the following statement which is a discrete analog of the theorem that the total curvature of a spacial closed curve is not less that $2\pi$ (see, e.g.,~\cite{Fenchel}):

\emph{The sum of the exterior angles of an $r$-gon in Euclidean space $\R^q$ is at least~$2\pi$.}

Let $a_1,\dots,a_r$ be oriented vectors representing the edges of the polygon. The Gauss map sends these vectors to some points $A_1,\dots,A_r$ on $S^{q-1}$, and the great circle segment connecting~$A_i$ and~$A_{i+1}$ has length equal to the exterior angle at the vertex connecting $a_i$ and $a_{i+1}$ (here the indices are cyclic). Thus it is enough to show that the perimeter of this spherical polygon $(A_1,\dots,A_r)$ is at least~$2\pi$. By the Crofton formula, for this it is enough to show that the polygon intersects every great sphere~$S^{q-2}$.

Choose a great sphere $S = S^{q-2} \subset S^{q-1}$ and let $P = \R^{q-1} \subset \R^{q}$ be the corresponding hyperplane. If some vector $a_i$ is contained in $P$, then $S$ contains $A_i$. Otherwise, translate $P \subset \R^q$ so that it sits as a supporting hyperplane to the polygon $(a_1,\dots,a_r)$, say at the vertex connecting edge $a_i$ to edge $a_{i+1}$. In this case the points $A_i$ and $A_{i+1}$ are separated by the great sphere $S$, so the great circle segment connecting $A_i$ to $A_{i+1}$ must intersect~$S$.
\end{proof}

\subsection{Critical points of the restricted length function}

With the proof of Lemma \ref{lem:delta} complete, we turn our attention to the fourth item of Proposition~\ref{prop:main}, that if some $r$-gon $x = (x_1,\dots,x_r) \in \partial \GG_\varepsilon(M,r)$ is a critical point of $\Lambda \big|_{\partial \GG_\varepsilon(M,r)}$, then ${\rm d}\Lambda_x$ is positive on inward vectors. Since $\partial \GG_\varepsilon(M,r)$ is a level set of $F$, the differential ${\rm d}F$ vanishes precisely on vectors tangent to $\partial \GG_\varepsilon(M,r)$. Then, since ${\rm d}F_x = t\,{\rm d} \Lambda_x$ on $T_x\GG_\varepsilon(M,r)$, the sign of ${\rm d}\Lambda_x$ is the same for all inward pointing tangent vectors at point~$x$. Therefore, it is enough to show that ${\rm d}\Lambda_x$ is positive on a single inward vector $V \in T_x\GG_\varepsilon(M,r)$.

For the remainder of this section we let $x = (x_1,\dots,x_r)$ be a critical point as discussed above, and we let $u_i, v_i \in T_{x_i}U$ be unit vectors tangent to the geodesic rays $x_{i-1}x_i$ and $x_ix_{i+1}$. We will drop the subscript on the differentials ${\rm d}\Lambda$ and ${\rm d}F$, since we will only be discussing them at the point $x$. We will continue to use the notation $n_i$ for the outward-pointing unit normal vector at~$x_i$ and $p_i = D_{n_i}$. We will use $\ell_i$ to represent the Finsler distance $f(x_{i-1},x_i)$.

We recall the differentials:
\[
{\rm d}\Lambda = (D_{u_1} - D_{v_1}, \dots, D_{u_r} - D_{v_r}),
\]
and
\[
\tfrac12 {\rm d}(\ln F) = \left( \frac{D_{u_1}}{\ell_1} - \frac{D_{v_1}}{\ell_2}, \dots, \frac{D_{u_r}}{\ell_r} - \frac{D_{v_r}}{\ell_1} \right).
\]

To show that there exists $V \in T_x\GG_\varepsilon(M,r)$ for which both differentials are positive, it is enough to find a vector $\nu \in T_{x_i}M$, for some $i$, such that $D_{u_i}(\nu) > 0$ and $D_{v_i}(\nu) < 0$. Indeed, we can then let $V \in T_x\GG_\varepsilon(M,r)$ be the vector whose only nonzero component is $\nu$.

The structure of this section is to assume that no such $\nu$ exists and analyze the consequences. The following lemma is the first step of this process.

\begin{Lemma}
\label{lem:samesign}
Suppose that for some fixed $i$, there is no $\nu \in T_{x_i}M$ satisfying $D_{u_i}(\nu) > 0$ and $D_{v_i}(\nu) < 0$. Then there exist $a_i$ and $b_i$, both positive or both negative, such that
\begin{equation}
\label{eqn:abrelation}
a_iD_{u_i} - b_iD_{v_i} = p_i.
\end{equation}
\end{Lemma}

\begin{proof}
The hypothesis implies that $\operatorname{ker}(D_{u_i}) \cap T_{x_i}M = \operatorname{ker}(D_{v_i}) \cap T_{x_i}M$, hence the covectors~$D_u$ and~$D_v$ are positive-proportional when restricted to $T_{x_i}M$. That is, there exist $a_i$ and $b_i$, both positive or both negative, such that either
\begin{equation*}
a_iD_{u_i} - b_iD_{v_i} = 0 \qquad {\rm or}\qquad
a_iD_{u_i} - b_iD_{v_i} = p_i.
\end{equation*}
The former is impossible since it implies that $u_i = v_i$.
\end{proof}

It follows from Lemma \ref{lem:samesign} that the three covectors $D_{u_i}$, $D_{v_i}$, and $p_i$ span a $2$-plane $P$, hence there exists a unique vector $w_i \in T_{x_i}M \cap I$ such that $D_{w_i} \in P$ and $D_{u_i}(w_i) > 0$. In the next lemma we develop a universal constant, which allows us to compare $D_{u_i}$ with $p_i(u_i)$ and $D_{v_i}$ with $p_i(v_i)$, for any $u_i$ and $v_i$ satisfying a reflection law as in (\ref{eqn:abrelation}). We observe the similarities in (\ref{eqn:abrelation}) to both the Finsler reflection law and the $F$-reflection law, and we note that $w_i$ is an analogue of the vector $w$ introduced after Lemma \ref{lem:tangentlaw} for the $F$-billiard reflection.

\begin{Lemma}
\label{lem:trig}
There exist constants $C>0$ and $m \in \mathbb{N}$, such that, for all $x \in M$, and for all $w \in T_xM \cap I$, if $P = \operatorname{Span}\left\{p,D_w\right\}$, and if $u \in I$ is any vector with $D_u \in P$ and $D_u(w) > 0$, then
\[
\frac{1}{C}p(u)^{2m} \leq 1 - D_u(w) \leq Cp(u)^{2m}.
\]
\end{Lemma}

\begin{proof}
By compactness of $M$ and $T_xM \cap I$ it is enough to show the existence of $C$ at a single point $x$ and for a single $2$-plane $P$. Let $\Gamma \subset I$ be the open curve consisting of those points $u$ described in the statement of the lemma. We claim that $p|_\Gamma \colon \Gamma \longrightarrow \R$ is injective. Indeed, suppose that $p(u) = p(v)$ for $u,v \in \Gamma$. We may write $aD_u - bD_v = p$; here $a$ and $b$ have the same sign, since $D_u(w)$ and $D_v(w)$ are both positive and $p(w)=0$. Applying this relation to $u$ and $v$ yields
\[
a - bD_v(u) = p(u) = p(v) = aD_u(v) - b,
\]
 therefore $a(1 - D_u(v)) = -b(1 - D_v(u))$, contradicting that $a$ and $b$ have the same sign.

Now consider the map $h \colon p(\Gamma) \longrightarrow \R$ given by $p(u) \longmapsto D_u(w)$. Then $h(0) = D_w(w) = 1$, and $h'(0) = 0$ since $D_u(w)$ is maximized at $u = w$. Therefore the first nonzero derivative of $h$ has even order and is negative; in particular, $1 - D_u(w) \sim p(u)^{2m}$ for some $m$.
\end{proof}

It follows from quadratic convexity of the indicatrix that $m=1$, but we omit the details here. We are most interested in the following form of Lemma \ref{lem:trig}.

\begin{Corollary}
\label{cor:trig}
For any $x \in M$ and any $u, v \in I \subset T_xU$ such that $p(u) > 0$, $p(v) < 0$, and $aD_u - bD_v = p$, where $a$ and $b$ have the same sign, the following inequality holds:
\[
\frac{p(v)^{2m}}{C^2p(u)^{2m}} \leq \frac{1-D_v(w)}{1-D_u(w)} \leq C^2\frac{p(v)^{2m}}{p(u)^{2m}},
\]
where $w \in T_xM$ is the unique unit vector with $D_w \in \operatorname{Span}\left\{D_u, D_v\right\}$ and $D_u(w) > 0$.
\end{Corollary}

We will show the existence of an appropriate vector $\nu$ in two separate cases: the first, when all side lengths of the $r$-gon are sufficiently short; and the second, when there is at least one long side. The following lemma treats the first case.

\begin{Lemma}
\label{lem:eta}
There exists $\eta > 0$ such that, if $p_i(u_i) < \eta$ and $-p_i(v_i) < \eta$ for all $i \in \left\{1, \dots, r\right\}$, then there exists an index $i$ and a vector $\nu \in T_{x_i}M$, such that $D_{u_i}(\nu) > 0$ and $D_{v_i}(\nu) < 0$.
\end{Lemma}

\begin{proof} Assume that no such $\nu$ exists, then equation (\ref{eqn:abrelation}), with $a_i$ and $b_i$ both positive or both negative, holds at every vertex $i$. We will use the ``all small'' hypothesis of the lemma to arrive at a contradiction. We first focus on a single vertex and drop the subscript $i$. Apply (\ref{eqn:abrelation}) to vectors $u$ and $v$ to obtain
\begin{gather}\label{eqn:ab}
a - bD_v(u)  = p(u), \\
aD_u(v) - b  = p(v).\label{eqn:ab2}
\end{gather}
Suppose that $\eta = \frac{\delta}{2+2K}$, where $\delta > 0$ is a small unspecified number and $K \geq 1$ is the constant determined in Lemma~\ref{lem:upperbound}. We claim that $D_u(v)$ and $D_v(u)$ cannot be negative.

First, if $a$ and $b$ are negative, then $D_u(v)$ and $D_v(u)$ are both positive because $p(u)>0$, $p(v)<0$. Otherwise, if $a$ and $b$ are positive, and also if $D_v(u)$ is negative, then{\samepage
\begin{gather*}
a  = p(u) + bD_v(u) < p(u), \\
b  = -p(v) + a D_u(v) < -p(v) + p(u)D_u(v) < -p(v) + p(u),
\end{gather*}
and so $a$ and $b$ are both less than $2\eta$. The same bounds hold if $D_u(v)$ is negative.}

Now apply (\ref{eqn:abrelation}) to the vector $n$ to obtain
\[
1 = aD_u(n) - bD_v(n) < a - bD_v(n) < 2\eta + bK < 2\eta(1 + K) = \delta,
\]
contradicting the assumption that $D_v(u)$ or $D_u(v)$ is negative. Now, solve for $a$ and $b$ using equations (\ref{eqn:ab}) and (\ref{eqn:ab2}) to obtain
\begin{gather*}
a  = \frac{p(u) - p(v)D_v(u)}{1 - D_u(v)D_v(u)} < \frac{2\eta}{1 - D_u(v)D_v(u)}, \\
b  = \frac{-p(v) + p(u)D_u(v)}{1 - D_u(v)D_v(u)} < \frac{2\eta}{1 - D_u(v)D_v(u)}.
\end{gather*}
Again we apply (\ref{eqn:abrelation}) to the vector $n$ to obtain
\[
1 = aD_u(n) - bD_v(n) < \frac{2\eta(1+K)}{1 - D_u(v)D_v(u)} = \frac{\delta}{1 - D_u(v)D_v(u)}.
\]
Therefore, $1 - D_u(v)D_v(u) = \delta$, and so both $D_u(v) \geq 1 - \delta$ and $D_v(u) \geq 1 - \delta$. It follows that $u$ and $v$ are sufficiently close.

The conclusion of the proof follows the lines of the discussion in Section~\ref{sec:lemdeltaproof} (the proof of Lemma \ref{lem:delta}). Namely, the above analysis holds at every vertex~$i$, and this contradicts the existence of a large exterior angle in an auxiliary Euclidean metric.
\end{proof}

We are now equipped to prove Proposition~\ref{prop:main}.

\subsection{The proof of Proposition \ref{prop:main}}

\begin{proof}[Proof of Proposition~\ref{prop:main}]
The first two items follow immediately from Lemma~\ref{lem:delta}. The argument for the third item is similar. In particular, we apply the same logic of Lemma~\ref{lem:delta} to the length function~$\Lambda$, instead of the function $F$, to obtain a similar statement: if one edge of an $r$-periodic Finsler billiard trajectory is short, then so are all its edges. Therefore there exists some $\varepsilon > 0$ such that all critical points of~$\Lambda$ are contained in the interior of $\GG_\varepsilon(M,r)$.

Let $\eta$ be a fixed number satisfying Lemma~\ref{lem:eta}, and let $R$ and $\rho$ be the constants from Lemma~\ref{lem:finslerbound}. To satisfy the fourth item of the proposition, choose $\varepsilon$ small enough so that \mbox{$\eta > \frac{A}{\rho}\varepsilon^{\frac{1}{2r}}$}, where $A > 0$ is a constant we will specify shortly. Assume that all lengths $\ell_i$ satisfy $\ell_i < \eta\rho$. Then, by Lemma~\ref{lem:finslerbound}, all corresponding $u_i$ and $v_i$ are $\eta$-tangent, so the conditions of Lemma \ref{lem:eta} are satisfied, confirming the existence of an appropriate vector $\nu$. Otherwise, there exists an index $j$ such that $\ell_j \geq \eta\rho > A\varepsilon^{\frac{1}{2r}}$. Using the fact that the product of the squared lengths is $\varepsilon$, we write
\[
\varepsilon^{\frac12} = \prod_{1 \leq i \leq r} \ell_i > A\varepsilon^{\frac{1}{2r}} \prod_{i \neq j} \ell_i.
\]
Therefore, there exists an index $k$ such that $\ell_k < \varepsilon^{\frac{1}{2r}} A^{-\frac{1}{r-1}}$. It follows that
\[
A^\frac{r}{r-1} < \frac{\ell_j}{\ell_k} = \frac{\ell_j}{\ell_{j-1}} \cdot \frac{\ell_{j-1}}{\ell_{j-2}} \cdots \frac{\ell_{k+1}}{\ell_k},
\]
where the indices are understood cyclically. There are at most $r-1$ factors on the right side, therefore, there exists some index $i$ such that
\[
A^\frac{r}{(r-1)^2} < \frac{\ell_{i+1}}{\ell_i}.
\]
Now let $A = \big(\frac{R}{\rho}\big)^{\frac{(r-1)^2}{r}} \cdot C^{\frac{(r-1)^2}{mr}}$, where $C$ and $m$ are the constants determined in Lemma \ref{lem:trig}. Then we have
\[
\frac{R}{\rho} \cdot C^{\frac{1}{m}} < \frac{\ell_{i+1}}{\ell_i} \leq -\frac{R}{\rho}\frac{p_i(v_i)}{p_i(u_i)},
\]
where the second inequality is Lemma \ref{lem:finslerbound}. It follows that
\begin{gather}\label{eqn:shortlong}
1 < \frac{p_i(v_i)^{2m}}{C^2p_i(u_i)^{2m}}.
\end{gather}
Assume there is no vector $\nu \in T_{x_i}M$ such that $D_{u_i}(\nu) >0$ and $D_{v_i}(\nu) < 0$, since otherwise, we are done. Then the hypotheses of Corollary~\ref{cor:trig} hold, and we write
\[
\frac{1 - D_{v_i}(w)}{1 - D_{u_i}(w)} \geq \frac{p_i(v_i)^{2m}}{C^2p_i(u_i)^{2m}} > 1,
\]
where the second inequality is~(\ref{eqn:shortlong}). Therefore $D_{u_i}(w) > D_{v_i}(w)$, and since $\ell_{i+1} > \ell_i$, we also have $D_{u_i}(w)\ell_{i+1} > D_{v_i}(w)\ell_i$. Thus the vector $V \in T_x\GG_\varepsilon(M,r)$, whose only nonzero component is $w$, satisfies ${\rm d}\Lambda(V) > 0$ and ${\rm d}F(V) > 0$, as desired.
\end{proof}

\section{Topology of cyclic configuration spaces}

Let $r\geq 2$ be an integer, and let $M$ be a topological space.
The {\em ordered configuration space} of $r$ pairwise distinct points on $M$ is the space
\[
\FF(M,r)\coloneqq \big\{(x_1,\ldots,x_r)\in M^r \colon x_i\neq x_j\ \text{for all}\ i\neq j \big\}.
\]
The symmetric group on $r$ letters $\Sym_r$ acts (from the left) on $\FF(M,r)$ by permuting the points, that is, for a permutation $\pi\in\Sym_r$
\[
\pi\cdot (x_1,\ldots,x_r) = (x_{\pi(1)},\ldots,x_{\pi(r)}).
\]
The {\em unlabeled configuration space} of $r$ pairwise distinct points on $M$ is the orbit space
$\FF(M\!, r)\!{/}\!\Sym_r$. We refer to F.~Cohen \cite{Cohen1976LNM533} and Fadell \& Husseini~\cite{Fadell-Husseomo2001} for background on configuration spaces.

Let us repeat that the {\em ordered cyclic configuration space} of $r\geq 2$ consecutively distinct points on $M$ is the space
\[
\GG(M,r)\coloneqq\big\{(x_1,\ldots,x_r)\in M^r\colon x_i\neq x_{i+1}\ \text{for all}\ i\big\},
\]
where by convention $x_{r+1}=x_1$. Clearly, $\FF(M,r)\subseteq \GG(M,r)\subset M^r$. The Dihedral group $D_r=\langle a,b\colon a^r=b^2=1, ab=ba^{r-1}\rangle$ acts naturally on $\GG(M,r)$ by
\begin{gather*}
	a\cdot(x_1,x_2,\ldots,x_{r-1},x_r) = (x_r,x_1,\ldots,x_{r-2},x_{r-1}),\\
	b\cdot(x_1,x_2,\ldots,x_{r-1},x_r) = (x_r,x_{r-1},\ldots,x_2,x_1).
\end{gather*}
On the other hand, due to the geometric restriction coming from the Finsler distance being not symmetric, we consider only the action of the cyclic subgroup $\Z_r=\langle a\rangle$ on $\GG(M,r)$. Thus, in this paper, the {\em unlabeled cyclic configuration space} of $r$ consecutively distinct points on $M$ is the orbit space $\GG(M,r)/\Z_r$.

In this section we study the topology of the unlabeled cyclic configuration space $\GG\big(S^{d-1}\!,r\big)/\Z_r\!$ for $r$ a prime. First, using an appropriate spectral sequence, we determine the cohomology of the unlabeled configuration space with coefficients in the field $\F_r$.
\begin{theorem}	\label{th : topological--main--1}
	Let $r\geq 3$ be a prime.
	\begin{enumerate}\itemsep=0pt
		\item[\rm (1)] Let $d$ be an even integer, and let
		\begin{gather*}
		A\coloneqq\{\ell(d-2),\ell(d-2)+1\colon 1\leq\ell\leq r-2 \},
		\\
		B\coloneqq\{0,1,\ldots, (r-1)(d-2)+1\}{\setminus}A.
		\end{gather*}
		Then
		\[
		H^n\big(\GG\big(S^{d-1},r\big)/\Z_r;\F_r\big)=
		\begin{cases}
			\F_r\oplus\F_r, & n\in A,\\
			\F_r , & n\in B,\\
			0, & \text{\rm otherwise}.
		\end{cases}
		\]
		
		\item[\rm (2)] Let $d$ be an odd integer, and let
		\begin{gather*}
		C\coloneqq\big\{2\ell(d-2),2\ell(d-2)+1\colon 1\leq\ell\leq\tfrac{r-3}{2} \big\},
		\\
		D\coloneqq\{0,1,\ldots, (r-1)(d-2)+1\}{\setminus}C.
		\end{gather*}
		Then
		\[
		H^n\big(\GG\big(S^{d-1},r\big)/\Z_r;\F_r\big)=
		\begin{cases}
			\F_r\oplus\F_r, & n\in C,\\
			\F_r , & n\in D,\\
			0, & \text{\rm otherwise}.
		\end{cases}
		\]
	\end{enumerate}
\end{theorem}

Second, using the same spectral sequence we derive the following estimate of the Lusternik--Schnirelmann category, which can be also deduced from the work of Karasev \cite[Theorem~7]{Karasev2009}.

\begin{theorem}	\label{th : topological--main--2}
	Let $d\geq 3$ be an integer, and let $r\geq 3$ be a prime.
	Then the Lusternik--Schnirelmann category of the unlabeled cyclic configuration space $\GG\big(S^{d-1},r\big)/\Z_r$ is bounded from below as follows:
	\begin{equation}		\label{LS estimate}
			\cat \big(\GG\big(S^{d-1},r\big)/\Z_r\big)\geq (r-1)(d-2)+1.
	\end{equation}
\end{theorem}	

The estimate of the Lusternik--Schnirelmann category \eqref{LS estimate} of the unlabeled cyclic configuration space $\GG\big(S^{d-1},r\big)/\Z_r$, in the case when $r$ is a prime, is obtained by exhibiting an element of the cohomology $H^{(r-1)(d-2)+1}(\B(\Z_r);\F_r)$ that does not vanish along the homomorphism
\[
\xymatrix{
H^*(\B(\Z_r);\F_r)\ar[rr]^-{p^*} & &H^*\big( \GG\big(S^{d-1},r\big)/\Z_r\big),
}
\]
which is induced by the unique, up to a homotopy, map $p\colon  \GG\big(S^{d-1},r\big)/\Z_r\longrightarrow\B(\Z_r)$, see Section~\ref{proof of Thm. 3.2}. For this we substitute the map $p$ with the projection map $\pi_G$ of the fibration
\[
\xymatrix{
\E(\Z_r)\times_{\Z_r} \GG\big(S^{d-1},r\big)\ar[rr]^-{\pi_G} & & \B(\Z_r)
}
\]
and study the correspond Serre spectral sequence $E_*^{*,*}\big(\E(\Z_r)\times_{\Z_r} \GG\big(S^{d-1},r\big)\big)$, see Section~\ref{subsec : Setting up a spectral sequence}. The differentials of this spectral sequence are analyzed in Section~\ref{sec : Computing the spectral sequence} using the comparison of spectral sequences induced by the morphism of fiber bundles
\[
	\xymatrix{
	 \E(\Z_r)\times_{\Z_r} \FF\big(\R^{d-1},r\big)\ar[d]^{\pi_F}\ar[rr]^-{\id\times_{\Z_r}i} & & \E(\Z_r)\times_{\Z_r} \GG\big(S^{d-1},r\big)\ar[d]^{\pi_G}\\
	 \B(\Z_r)\ar[rr]^-{\id} & & \B(\Z_r).
	}
\]
In this way we identify an element of the cohomology $u\in H^{(r-1)(d-2)+1}(\Z_r;\F_r)$ with the property that $\pi_G^*(u)\neq 0$, and consequently $p^*(u)\neq 0$.
Using the classical concept of category weight of an element of cohomology, and it properties summarized in Lemma \ref{lem:LS-2}, we establish the lower bound \eqref{LS estimate}.

\subsection{Setting up a spectral sequence}\label{subsec : Setting up a spectral sequence}

Now we set up a spectral sequence that converges to the cohomology of the unlabeled cyclic configuration space $H^*\big(\GG\big(S^{d-1},r\big)/\Z_r;\F_r\big)$. We use the fact that the group $\Z_r$ acts freely on the cyclic configuration space $\GG\big(S^{d-1},r\big)$.

Thus, since $\Z_r$ acts freely on $\GG\big(S^{d-1},r\big)$, the Borel construction $\E(\Z_r)\times_{\Z_r} \GG\big(S^{d-1},r\big)$ and the quotient space $\GG\big(S^{d-1},r\big)/\Z_r$ are homotopy equivalent.
Indeed, the map
\[
\E(\Z_r)\times_{\Z_r} \GG\big(S^{d-1},r\big)\longrightarrow \GG\big(S^{d-1},r\big)/\Z_r,
\]
induced by the $\Z_r$-equivariant projection on the second factor $\E(\Z_r)\times \GG\big(S^{d-1},r\big)\!\longrightarrow\! \GG\big(S^{d-1},r\big),\!$ is a fibration with a contractible fiber $\E(\Z_r)$, and therefore a homotopy equivalence. Hence,
\[
H^*\big(\GG\big(S^{d-1},r\big)/\Z_r;\F_r\big)\cong H^*\big(\E(\Z_r)\times_{\Z_r} \GG\big(S^{d-1},r\big);\F_r\big),
\]
and so we compute the cohomology of the Borel construction instead.
An advantage of the Borel construction is that it is the total space of the following fibration
\begin{equation}
\label{eq : fibration}
\xymatrix{
\GG\big(S^{d-1},r\big)\ar[r] & \E(\Z_r)\times_{\Z_r} \GG\big(S^{d-1},r\big)\ar[r]^-{\pi_G} & \B(\Z_r),
}	
\end{equation}
where $\pi_G$ is induced by the $\Z_r$-equivariant projection on the first factor $\E(\Z_r)\times \GG\big(S^{d-1},r\big)\longrightarrow \E(\Z_r)$.

The Serre spectral sequence induced by the fibration \eqref{eq : fibration} has the $E_2$-term given by
\begin{align}
E_2^{p,q}\big(\E(\Z_r)\times_{\Z_r} \GG\big(S^{d-1},r\big)\big) &=  \mathcal{H}^p\big(\B(\Z_r); H^q\big(\GG\big(S^{d-1},r\big);\F_r\big)\big)\nonumber\\
&\cong  H^p\big(\Z_r; H^q\big(\GG\big(S^{d-1},r\big);\F_r\big)\big)\label{eq : E2 term}.
\end{align}
Here $\mathcal{H}^*(\cdot)$ indicates that the we have the cohomology with local coefficients; for more details consult for example \cite[Section~3.H]{Hatcher2002}.
The local system is determined by the action of the fundamental group of the base space $\pi_1(\B(\Z_r))\cong\Z_r$ on the cohomology of a fiber of the fibration~\eqref{eq : fibration}.
The second notation assumes that the coefficients we have for the group cohomology of $\Z_r$ are given in the $\Z_r$-module $H^q\big(\GG\big(S^{d-1},r\big);\F_r\big)$.
In the case when $\F_r$ is a trivial $\Z_r$-module we set $H^*(\Z_r;\F_r)=\F_r[t]\otimes\Lambda(e)$ where $\deg(t)=2$, $\deg(e)=1$, and $\Lambda(\cdot)$ denotes the exterior algebra.

All the cohomologies we work with are $\F_r$ vector spaces and therefore the Serre spectral sequence induced by the fibration~\eqref{eq : fibration} converges to the cohomology $H^*\big(\E(\Z_r)\times_{\Z_r} \GG\big(S^{d-1},r\big);\F_r\big)$ as a vector space, that is
\[
H^n\big(\GG\big(S^{d-1},r\big)/\Z_r;\F_r\big)\cong H^n\big(\E(\Z_r)\times_{\Z_r} \GG\big(S^{d-1},r\big);\F_r\big)\cong \bigoplus_{p+q=n}E^{p,q}_{\infty},
\]
for every integer $n\geq 0$. In turn, since $\GG\big(S^{d-1},r\big)/\Z_r$ is an open $(r(d-1))$-manifold and has no cohomology in dimensions $\geq r(d-1)$, we have that $E^{p,q}_{\infty}=0$ for all $p+q\geq r(d-1)$.
For more details about Serre spectral sequences consult for example \cite[Chapters~5--6]{McCleary2001} or \cite[Chapter~3]{Fomenko2016}.

The first ingredient in the computation of the $E_2$-term of the spectral sequence~\eqref{eq : E2 term} is a~description of the cohomology of the fiber $H^q\big(\GG\big(S^{d-1},r\big);\F_r\big)$.
For that we use the results of Farber \cite[Theorems~18 and~19]{Farber2002-II}.

\begin{Proposition}	\label{prop : cohomology of unlabeled ring}
	Let $d\geq 3$ be an integer, and let $r$ be an odd prime.
	\begin{enumerate}\itemsep=0pt
	\item[\rm (1)] If $d$ is even, then the cohomology ring $H^*\big(\GG\big(S^{d-1},r\big);\F_r\big)$ of the ordered cyclic configuration space $\GG\big(S^{d-1},r\big)$ is generated by the elements
	\[
	\alpha, \
	\beta_1, \
	\beta_2, \ \ldots, \
	\beta_{r-2},
	\]
	of degrees
	\begin{gather*}
	\deg(\alpha)=d-1, \ \deg(\beta_1)=d-2, \ \deg(\beta_2)=2(d-2), \ \ldots , \\ \deg(\beta_{r-2})=(r-2)(d-2),
	\end{gather*}
	subject to the relations
	\[
	\alpha^2=0,  \qquad
	\beta_i\beta_j=
	\begin{cases}
		\dfrac{(i+j)!}{i!\,j!}\beta_{i+j}, & \text{for } i+j\leq r-2,\\
		0,								 & \text{otherwise},
	\end{cases}
	\]
	for $1\leq i\leq j\leq r-2$.
	In particular, for every $1\leq k\leq r-2$ we have that $\beta_k=a_k\cdot\beta_1^k$ where $a_k\in\F_r{\setminus}\{0\}$.
	\item[\rm (2)] If $d$ is odd, then the cohomology ring $H^*\big(\GG\big(S^{d-1},r\big);\F_r\big)$ of the ordered cyclic configuration space $\GG\big(S^{d-1},r\big)$ is generated by the elements
	\[
	\gamma, \
	\delta_1, \
	\delta_2, \ \ldots, \
	\delta_{\frac{r-3}2},
	\]
	of degrees
	\begin{gather*}
	\deg(\gamma)=2(d-2)+1, \ \deg(\delta_1)=1\cdot 2(d-2), \ \deg(\delta_2)=2\cdot 2(d-2), \ \ldots , \\ \deg\big(\delta_{\frac{r-3}2}\big)=\tfrac{r-3}{2}\cdot 2(d-2),
	\end{gather*}
	subject to the relations
	\[
	\gamma^2=0,   \qquad
	\delta_i\delta_j=
	\begin{cases}
		\dfrac{(i+j)!}{i!\,j!}\delta_{i+j}, & \text{for } i+j\leq \tfrac{r-3}2,\\
		0,								 & \text{otherwise},
	\end{cases}
	\]
	for $1\leq i\leq j\leq \tfrac{r-3}2$. In particular, for every $1\leq k\leq \tfrac{r-3}2$ we have that $\delta_k=b_k\cdot\delta_1^k$ where $b_k\in\F_r{\setminus}\{0\}$.
	\end{enumerate}
\end{Proposition}

Transforming the previous information on the cohomology ring of the ordered cyclic con\-fi\-guration space $\GG\big(S^{d-1},r\big)$ into an additive language we get the following corollary.

\begin{Corollary}	\label{cor : cohomology of unlabeled additive}
	Let $d\geq 3$ be an integer, and let $r$ be an odd prime.
	\begin{enumerate}\itemsep=0pt
	\item[\rm (1)] If $d$ is even and $A=\{\ell(d-2),\ell(d-2)+1 \colon 1\leq\ell\leq r-2 \}\cup\{0,(r-1)(d-2)+1\}$, then
	\[
	H^n\big(\GG\big(S^{d-1},r\big);\F_r\big)=
		\begin{cases}
			\F_r, & n\in A,\\
			0, & \text{\rm otherwise}.
		\end{cases}
	\]
	\item[\rm (2)] If $d$ is odd and $B=\big\{2\ell(d-2),2\ell(d-2)+1 \colon 1\leq\ell\leq \tfrac{r-3}2 \big\}\cup\{0,(r-1)(d-2)+1\}$, then
	\[
	H^n\big(\GG\big(S^{d-1},r\big);\F_r\big)=
		\begin{cases}
			\F_r, & n\in B,\\
			0, & \text{\rm otherwise}.
		\end{cases}
	\]
	\end{enumerate}
\end{Corollary}

From Proposition \ref{prop : cohomology of unlabeled ring} and its consequence Corollary~\ref{cor : cohomology of unlabeled additive} we can detect the action of the fundamental group of the base space $\pi_1(\B(\Z_r))\cong\Z_r$ on the cohomology of the ordered cyclic configuration space and compute the $E_2$-term of the Serre spectral sequence~\eqref{eq : E2 term}; for an illustration of the $E_2$-term in the case of even $d$ see Fig.~\ref{fig : 02}.

\begin{figure}[ht!]\centering
\includegraphics[scale=0.8]{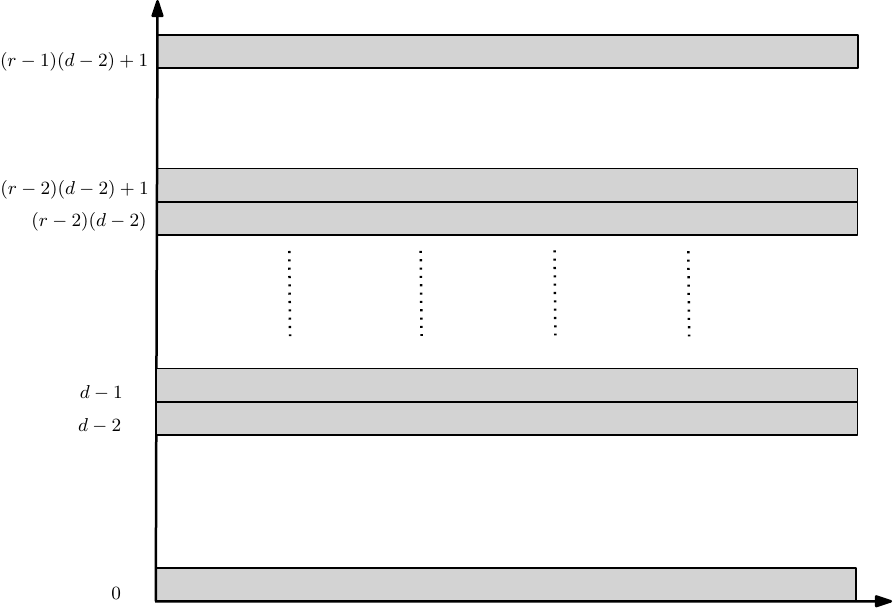}
\caption{\small The $E_2$-term of the Serre spectral sequence~\eqref{eq : E2 term} in the case when $d\geq 4$ is even.}\label{fig : 02}
\end{figure}

\begin{Corollary}	\label{cor : E_2 term }
		Let $d\geq 3$ be an integer, and let $r$ be an odd prime.
		The $E_2$-term of the Serre spectral sequence~\eqref{eq : E2 term} has the following description.
	\begin{enumerate}\itemsep=0pt
	\item[\rm (1)] If $d$ is even and $A=\{\ell(d-2),\ell(d-2)+1 \colon 1\leq\ell\leq r-2 \}\cup\{0,(r-1)(d-2)+1\}$, then
	\begin{align*}
	E_2^{p,q}&\cong   H^p\big(\Z_r; H^q\big(\GG\big(S^{d-1},r\big);\F_r\big)\big)	\\
			 &\cong   H^p(\Z_r;\F_r)\otimes H^q\big(\GG\big(S^{d-1},r\big);\F_r\big)\cong
		\begin{cases}
			H^p(\Z_r;\F_r), & n\in A,\\
			0, & \text{\rm otherwise}.
		\end{cases}
	\end{align*}

	\item[\rm (2)] If $d$ is odd and $B=\big\{2\ell(d-2),2\ell(d-2)+1 \colon 1\leq\ell\leq \tfrac{r-3}2 \big\}\cup\{0,(r-1)(d-2)+1\}$, then
	\begin{align*}
	 E_2^{p,q} &\cong   H^p\big(\Z_r; H^q\big(\GG\big(S^{d-1},r\big);\F_r\big)\big)	\\
	 		 &\cong   H^p(\Z_r;\F_r)\otimes H^q\big(\GG\big(S^{d-1},r\big);\F_r\big)\cong
		\begin{cases}
			H^p(\Z_r;\F_r), & n\in B,\\
			0, & \text{\rm otherwise}.
		\end{cases}
	\end{align*}
	\end{enumerate}
\end{Corollary}	

\begin{proof}
In the fibration \eqref{eq : fibration}, the cohomology groups of the fiber $H^q\big(\GG\big(S^{d-1},r\big);\F_r\big)$ are isomorphic either to $0$ or to $\F_r$.
Thus the group $\pi_1(\B(\Z_r))\cong\Z_r$ can only act trivially on $H^q\big(\GG\big(S^{d-1},r\big);\F_r\big)$:
Indeed, if $f\colon \F_r\longrightarrow \F_r$ is an $\F_r$-linear map induced by the action of a~ge\-ne\-rator of $\Z_r$, and $f(1)=a\in \F_r$, then by Fermat's little theorem $1=\id_{\F_r}(1)=f^r(1)=a^r=a$ and so $f=\id_{\F_r}$. Consequently, the cohomology $H^q\big(\GG\big(S^{d-1},r\big);\F_r\big)$ is a trivial $\Z_r$-module for every~$q$, and the statement of the corollary follows.
\end{proof}

\begin{Remark}It is important to point out that the previous corollary also implies that the differentials of the Serre spectral sequence~\eqref{eq : E2 term} are not only $H^*(\Z_r;\F_r)$-module morphisms, but even more, they satisfy the Leibniz rule; see \cite[Definition~1.6 and Proposition~5.6]{McCleary2001}.
\end{Remark}

\subsection{Computing the spectral sequence}\label{sec : Computing the spectral sequence}

In this section, before making explicit computation, we consider the ordered configuration space $\FF\big(\R^{d-1},r\big)$ as a subset of the ordered cyclic configuration space $\GG\big(S^{d-1},r\big)$.
The $\Z_r$-equivariant inclusion $i\colon \FF\big(\R^{d-1},r\big)\longrightarrow\GG\big(S^{d-1},r\big)$ induces the following morphism of Borel construction fibrations:
\begin{equation}\label{diag : morphism of Borel constructions}
\begin{split}&
	\xymatrix{
	 \E(\Z_r)\times_{\Z_r} \FF\big(\R^{d-1},r\big)\ar[d]^{\pi_F}\ar[rr]^-{\id\times_{\Z_r}i} & & \E(\Z_r)\times_{\Z_r} \GG\big(S^{d-1},r\big)\ar[d]^{\pi_G}\\
	 \B(\Z_r)\ar[rr]^-{\id} & & \B(\Z_r).
	}\end{split}
\end{equation}
The morphism of fibrations~\eqref{diag : morphism of Borel constructions} induces a morphism of associated Serre spectral sequences
\begin{equation}
	\label{eq : morphism of SSS}
	\xymatrix{
	E^{*,*}_{*}(\id\times_{\Z_r}i)\colon \ E^{*,*}_{*}\big(\E(\Z_r)\times_{\Z_r} \FF\big(\R^{d-1},r\big)\big) &
	E^{*,*}_{*}\big(\E(\Z_r)\times_{\Z_r} \GG\big(S^{d-1},r\big)\big),\ar[l]
	}
\end{equation}
which is the identity homomorphism on the zero row of the $E_2$-term, that is $E^{*,0}_{2}(\id\times_{\Z_r}i)=\id$.

The Serre spectral sequence $E^{*,*}_{*}\big(\E(\Z_r)\times_{\Z_r} \FF\big(\R^{d-1},r\big)\big)$ has been completely described first by F.~Cohen in his seminal paper \cite[Sections~8--10]{Cohen1976LNM533}, and much later, from the point of view of the equivariant Goresky--MacPherson formula, in \cite[Theorem~6.1]{Blagojevic2015}. For an illustration of the $E_2$-term of this spectral sequence see Fig.~\ref{fig : 03}. In particular, we will use the following facts proved in \cite{Blagojevic2015, Cohen1976LNM533}.

\begin{Proposition}	\label{prop : SS-Cohen}Let $d\geq 3$ be an integer, and let $r$ be an odd prime. The Serre spectral sequence associated to the Borel construction fibration
\begin{equation}	\label{eq : fibration-conf}
	\xymatrix{
\FF\big(\R^{d-1},r\big)\ar[r] & \E(\Z_r)\times_{\Z_r} \FF\big(\R^{d-1},r\big)\ar[r]^-{\pi_F} & \B(\Z_r),
}	
\end{equation}
with the $E_2$-term
\[
E^{p,q}_2=\mathcal{H}^p\big(\B(\Z_r);H^q\big(\FF\big(\R^{d-1},r\big);\F_r\big)\big)\cong H^p\big(\Z_r;H^q\big(\FF\big(\R^{d-1},r\big);\F_r\big)\big)
\]
has to following properties:
\begin{enumerate}\itemsep=0pt
\item[\rm (1)] For all $p\geq 1$ and $1\leq q \leq (r-1)(d-2)-1$, we have that
\[
E^{p,q}_2\big(\E(\Z_r)\times _{\Z_r}\FF\big(\R^{d-1},r\big)\big)=E^{p,q}_{\infty}\big(\E(\Z_r)\times _{\Z_r}\FF\big(\R^{d-1},r\big)\big)=0.
\]

\item[\rm (2)] For all $p\in\Z$ and $q\geq (r-1)(d-2)+1$, we have that
\[
E^{p,q}_2\big(\E(\Z_r)\times _{\Z_r}\FF\big(\R^{d-1},r\big)\big)=E^{p,q}_{\infty}\big(\E(\Z_r)\times _{\Z_r}\FF\big(\R^{d-1},r\big)\big)=0.
\]

\item[\rm (3)] For all $p\geq 0$ and $2\leq s \leq (r-1)(d-2)$, the differential
\[
\partial_{s}\colon \
E_{s}^{p,s-1}\big(\E(\Z_r)\times _{\Z_r}\FF\big(\R^{d-1},r\big)\big)\longrightarrow E_{s}^{p+s,0}\big(\E(\Z_r)\times_{\Z_r}\FF\big(\R^d,r\big)\big)
\]
vanishes.
Consequently, for $0\leq p \leq (r-1)(d-2)$, we have that
\[
H^p(\Z_r;\F_r) = E^{p,0}_2\big(\E(\Z_r)\times _{\Z_r}\FF\big(\R^{d-1},r\big)\big) = E_{\infty}^{p,0}\big(\E(\Z_r)\times _{\Z_r}\FF\big(\R^{d-1},r\big)\big).
\]

\item[\rm (4)] The only non-zero differential of the spectral sequence is
\begin{gather*}
\partial_{(d-2)(r-1)+1}\colon \
E_{(d-2)(r-1)+1}^{p,(d-2)(r-1)}\big(\E(\Z_r)\times _{\Z_r}\FF\big(\R^{d-1},r\big)\big)\\
\hphantom{\partial_{(d-2)(r-1)+1}\colon}{} \ \longrightarrow E_{(d-2)(r-1)+1}^{p+(d-2)(r-1)+1,0}\big(\E(\Z_r)\times_{\Z_r}\FF\big(\R^{d-1},r\big)\big),
\end{gather*}
and is an isomorphism for all $p\geq 0$.
Consequently, for all $p\geq (r-1)(d-2)+1$, one has
\[
E_{(d-2)(r-1)+2}^{p,0}\big(\E(\Z_r)\times _{\Z_r}\FF\big(\R^{d-1},r\big)\big)\cong E_{\infty}^{p,0}\big(\E(\Z_r)\times _{\Z_r}\FF\big(\R^{d-1},r\big)\big)=0.
\]
\end{enumerate}
\end{Proposition}

\begin{figure}[ht!]
\centering
\includegraphics[scale=0.8]{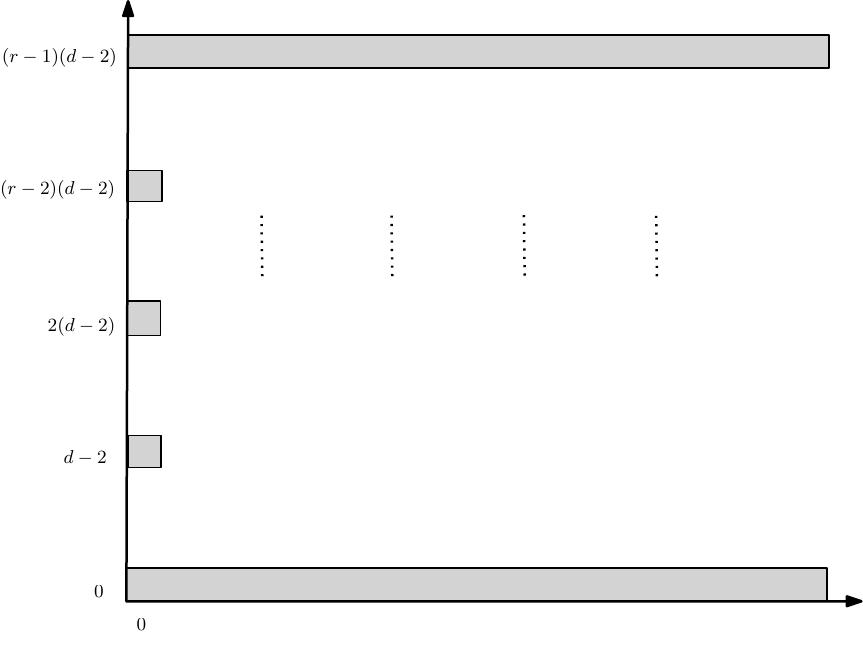}
\caption{\small The $E_2=E_{(r-1)(d-2)}$-term of the Serre spectral sequence of the fibration \eqref{eq : fibration-conf}.}\label{fig : 03}
\end{figure}

From the existence of the morphism $E^{*,*}_{*}(\id\times_{\Z_r}i)$ of the Serre spectral sequences~\eqref{eq : morphism of SSS}, the fact that $E^{*,0}_{2}(\id\times_{\Z_r}i)$ is the identity, and the following equality from Proposition~\ref{prop : SS-Cohen}:
\[
E_{\infty}^{p,0}\big(\E(\Z_r)\times _{\Z_r}\FF\big(\R^{d-1},r\big)\big) = E^{p,0}_2\big(\E(\Z_r)\times _{\Z_r}\FF\big(\R^{d-1},r\big)\big)=H^p(\Z_r;\F_r),
\]
which holds for $0\leq p \leq (r-1)(d-2)$, we deduce an important property of the Serre spectral sequence $E_*^{*,*}\big(\E(\Z_r)\times_{\Z_r} \GG\big(S^{d-1},r\big)\big)$ stated in the next corollary.

\begin{Corollary}\label{cor : differentials in zero row}
	Let $d\geq 3$ be an integer, and let $r$ be an odd prime.	For $0\leq p\leq (r-1)(d-2)$, one has
\[
E_{\infty}^{p,0}\big(\E(\Z_r)\times _{\Z_r}\GG\big(S^{d-1},r\big)\big) \cong E^{p,0}_2\big(\E(\Z_r)\times _{\Z_r}\GG\big(S^{d-1},r\big)\big) \cong H^p(\Z_r;\F_r).
\]
	In particular, all the differentials
\[
\partial_{s}\colon \
E_{s}^{p,s-1}\big(\E(\Z_r)\times _{\Z_r}\GG\big(S^{d-1},r\big)\big)\longrightarrow E_{s}^{p+s,0}\big(\E(\Z_r)\times_{\Z_r}\GG\big(S^{d-1},r\big)\big)
\]
vanish for $p\geq 0$ and $2\leq s \leq (r-1)(d-2)$.
\end{Corollary}

Combining the previous fact about the Serre spectral sequence $E_*^{*,*}\big(\E(\Z_r)\times_{\Z_r} \GG\big(S^{d-1},r\big)\big)$ with the observation that $E_{\infty}^{p,q}\big(\E(\Z_r)\times_{\Z_r} \GG\big(S^{d-1},r\big)\big)=0$ for all $p+q\geq r(d-1)$, we have that for some $s\geq (r-1)(d-2)+1$ the differential
\[
\partial_{s}\colon \
E_{s}^{p,s-1}\big(\E(\Z_r)\times _{\Z_r}\GG\big(S^{d-1},r\big)\big)\longrightarrow E_{s}^{p+s,0}\big(\E(\Z_r)\times_{\Z_r}\GG\big(S^{d-1},r\big)\big)
\]
does {\it not} vanish. In particular,
\[
E_{\infty}^{p,0}\big(\E(\Z_r)\times _{\Z_r}\GG\big(S^{d-1},r\big)\big)=0
\]
for all $p\geq r(d-1)$.

Now the properties of the Serre spectral sequence $E_*^{*,*}\big(\E(\Z_r)\times_{\Z_r} \GG\big(S^{d-1},r\big)\big)$, derived so far, will be used to completely compute it.
The computation proceeds in two separate steps, depending on the parity of~$d$.

\subsubsection[$d$ is an even integer, $d\geq 4$]{$\boldsymbol{d}$ is an even integer, $\boldsymbol{d\geq 4}$}\label{subsub -- 01}
The Serre spectral sequence $E^{*,*}_{*}\big(\E(\Z_r)\times_{\Z_r} \GG\big(S^{d-1},r\big)\big)$ has a multiplicative structure, and differentials satisfy the Leibniz rule.
Thus, to compute the spectral sequence completely it suffices to determine values of the differentials only on the generators of the cohomology ring of the fiber.
In this particular case, we need to evaluate the differentials on the elements $1\otimes\alpha,1\otimes\beta_1,\ldots, 1\otimes\beta_{r-1}$.
It is important to notice that
\[
(1\otimes\beta_1)^i=1\otimes\beta_1^i=a_i\cdot(1\otimes\beta_i)
\]
for $1\leq i\leq r-1$ and some $a_i\in\F_r{\setminus}\{0\}$.

\begin{figure}[ht!]\centering
\includegraphics[scale=0.8]{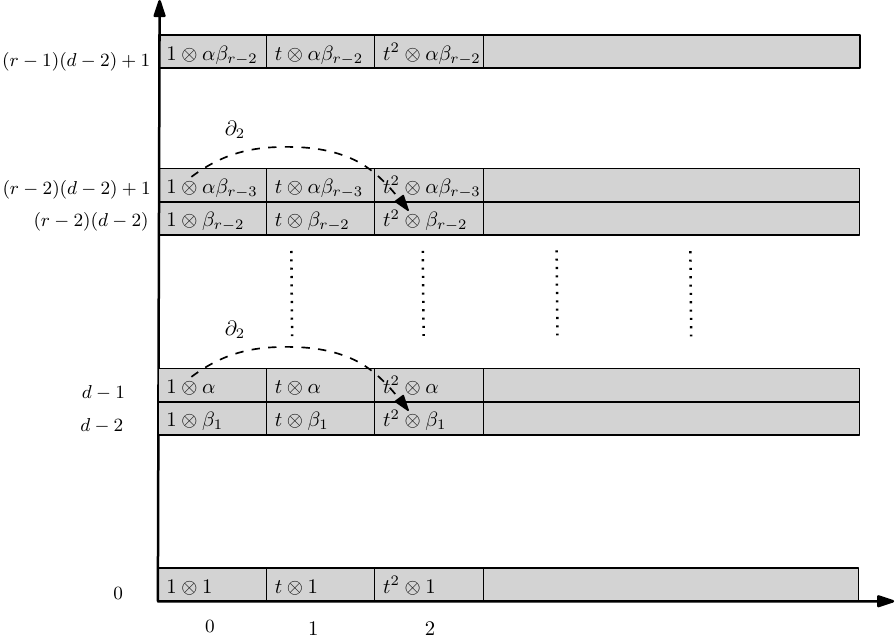}
\caption{\small For $d\geq 4$ even the $E_2$-term with differentials of the Serre spectral sequence \eqref{eq : E2 term}.}\label{fig : 04}
\end{figure}

Consider the differential $\partial_2$. The degrees of the cohomology generators $\beta_1,\ldots, \beta_{r-1}$ are $(d-2),\ldots,(r-2)(d-2)$, respectively. Since $H^i\big(\GG\big(S^{d-1},r\big);\F_r\big)=0$ vanishes for all $i\in\{(d-2)-1,\ldots,(r-2)(d-2)-1\}$, one has
\begin{equation}\label{eq : beta's}
\partial_2(1\otimes\beta_1)=\partial_2(1\otimes\beta_2)=\cdots=\partial_2(1\otimes\beta_{r-1})=0.	
\end{equation}
Thus, to completely determine the second differential we need to find the value
\[
\partial_2(1\otimes\alpha)\in E^{2,d-2}_{2}\big(\E(\Z_r)\times_{\Z_r} \GG\big(S^{d-1},r\big)\big).
\]

Let us assume that $\partial_2(1\otimes\alpha)=0$.
Then, due to the multiplicative property of the Serre spectral sequence, we have that the second differential $\partial_2$ vanishes, implying that
\[
E^{*,*}_{2}\big(\E(\Z_r)\times_{\Z_r} \GG\big(S^{d-1},r\big)\big)\cong E^{*,*}_{3}\big(\E(\Z_r)\times_{\Z_r} \GG\big(S^{d-1},r\big)\big).
\]
The vanishing of cohomology groups $H^*\big(\GG\big(S^{d-1},r\big);\F_r\big)$ in appropriate dimensions implies that the next possible non-zero differential is~$\partial_{d-1}$.
More precisely, $\partial_{d-1}(1\otimes\beta_1)$ might be non-zero while we know that $\partial_{d-1}(1\otimes\alpha)=0$.
Since $d-1\leq (r-1)(d-2)$, using Corollary~\ref{cor : differentials in zero row}, we get that $\partial_{d-1}(1\otimes\beta_1)=0$, and consequently the differential $\partial_{d-1}$ vanishes.
For the next differential we know that $\partial_{d}(1\otimes\beta_1)=0$, and since $d\leq (r-1)(d-2)$ from Corollary~\ref{cor : differentials in zero row}, we have that $\partial_{d}(1\otimes\alpha)=0$.
Hence, the differential~$\partial_{d}$ also vanishes.
In particular this means that
\[
E^{*,*}_{2}\big(\E(\Z_r)\times_{\Z_r} \GG\big(S^{d-1},r\big)\big)\cong  \cdots  \cong E^{*,*}_{d+1}\big(\E(\Z_r)\times_{\Z_r} \GG\big(S^{d-1},r\big)\big).
\]
Since for $s\geq d+1$ all differentials $\partial_s(1\otimes\beta_1)$ and $\partial_s(1\otimes\alpha)$ are zero, the multiplicative property of the Serre spectral sequence implies that all the differentials $\partial_s$ vanish. Thus
\[
E^{*,*}_{2}\big(\E(\Z_r)\times_{\Z_r} \GG\big(S^{d-1},r\big)\big)\cong E^{*,*}_{\infty}\big(\E(\Z_r)\times_{\Z_r} \GG\big(S^{d-1},r\big)\big).
\]
This is a contradiction to the fact, observed after Corollary~\ref{cor : differentials in zero row}, that at least one of the differentials
\[
\partial_{s}\colon \
E_{s}^{p,s-1}\big(\E(\Z_r)\times _{\Z_r}\GG\big(S^{d-1},r\big)\big)\longrightarrow E_{s}^{p+s,0}\big(\E(\Z_r)\times_{\Z_r}\GG\big(S^{d-1},r\big)\big),
\]
for $s\geq (r-1)(d-2)+1$, does not vanish. Therefore, $\partial_2(1\otimes\alpha)\neq 0$.

Set $\partial_2(1\otimes\alpha)=a\cdot (t\otimes\beta_1)$ where $a\in\F_r{\setminus}\{0\}$; for illustration see Fig.~\ref{fig : 04}.
The multiplicative property of the Serre spectral sequence, combined with~\eqref{eq : beta's}, yields that
\begin{align*}
\partial_2(1\otimes \alpha\beta_i)&=\partial_2((1\otimes \alpha) (1\otimes \beta_i))
 =\partial_2(1\otimes \alpha) (1\otimes \beta_i)+(-1)^{d-1}(1\otimes \alpha)\partial_2(1\otimes \beta_i)\\
&=\partial_2(1\otimes \alpha) (1\otimes \beta_i) =a\cdot (t\otimes\beta_1) (1\otimes \beta_i)\\
&=\begin{cases}
	\left(a\dfrac{(i+1)!}{i!1!}\right)\cdot (t\otimes\beta_{i+1}), & 1\leq i\leq r-3,\\
	0, & i=r-2.
\end{cases}
\end{align*}

\begin{figure}[ht]\centering
\includegraphics[scale=0.8]{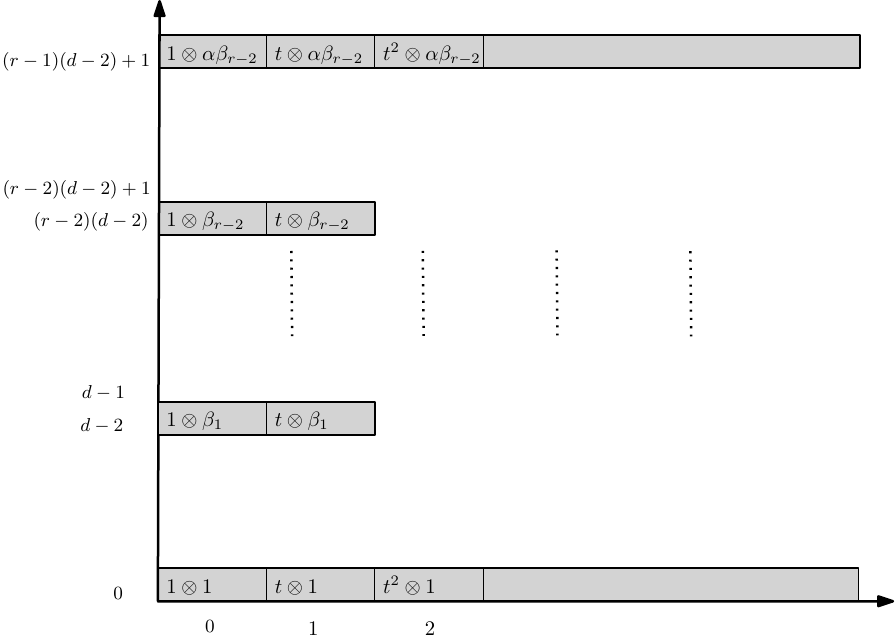}
\caption{\small For $d\geq 4$ even the $E_3$-term of the Serre spectral sequence \eqref{eq : E2 term}.}\label{fig : 05}
\end{figure}

Thus the $E_3$-term, illustrated in Fig.~\ref{fig : 05}, is given by
\begin{gather*}
E^{p,q}_{3}\big(\E(\Z_r)\times_{\Z_r} \GG\big(S^{d-1},r\big)\big)\\
\qquad{} \cong
\begin{cases}
	H^p(\Z_r;\F_r), & q\in\{0,(r-1)(d-2)+1\}, \\
	\F_r, & p\in\{0,1\},\,q\in\{(d-2),\ldots, (r-2)(d-2)\},\\
	0, & \text{otherwise}.
\end{cases}		
\end{gather*}
Moreover the multiplicative property of the Serre spectral sequence implies that all the differentials $\partial_3,\ldots,\partial_{(r-1)(d-2)+1}$ vanish.
Therefore
\[
E^{*,*}_{3}\big(\E(\Z_r)\times_{\Z_r} \GG\big(S^{d-1},r\big)\big)\cong   \cdots \cong E^{*,*}_{(r-1)(d-2)+2}\big(\E(\Z_r)\times_{\Z_r} \GG\big(S^{d-1},r\big)\big).
\]

The differential $\partial_{(r-1)(d-2)+2}$ is the only remaining differential that can be non-zero.
Since
\[
E^{p,0}_{2}\big(\E(\Z_r)\times_{\Z_r} \GG\big(S^{d-1},r\big)\big)\cong \cdots \cong E^{p,0}_{(r-1)(d-2)+2}\big(\E(\Z_r)\times_{\Z_r} \GG\big(S^{d-1},r\big)\big),
\]
and we know that
\[
E_{2}^{p,0}\big(\E(\Z_r)\times _{\Z_r}\GG\big(S^{d-1},r\big)\big)\neq 0 = E_{\infty}^{p,0}\big(\E(\Z_r)\times _{\Z_r}\GG\big(S^{d-1},r\big)\big)
\]
for all $p\geq r(d-1)$, we conclude that $\partial_{(r-1)(d-2)+2}\neq 0$.
The multiplicative property of the Serre spectral sequence again yields that $\partial_{(r-1)(d-2)+2}$ is completely determined by its image at $1\otimes\alpha\beta_{r-2}$, which has to be non-zero.
Set
\[
\partial_{(r-1)(d-2)+2}(1\otimes\alpha\beta_{r-2})=b\cdot\big(t^{\frac{(r-1)(d-2)+2}2}\otimes 1\big)
\]
 for some $b\in\F_r{\setminus}\{0\}$, as illustrated in Fig.~\ref{fig : 06}.
Consequently, the differentials
\begin{gather*}
\partial_{(r-1)(d-2)+2}\colon \
E_{s}^{p,(r-1)(d-2)+1}\big(\E(\Z_r)\times _{\Z_r}\GG\big(S^{d-1},r\big)\big)\\
\hphantom{\partial_{(r-1)(d-2)+2}\colon}{} \ \longrightarrow E_{s}^{p+(r-1)(d-2)+2,0}\big(\E(\Z_r)\times_{\Z_r}\GG\big(S^{d-1},r\big)\big)
\end{gather*}
are isomorphisms for all $p\geq 0$.

\begin{figure}[ht]\centering
\includegraphics[scale=0.8]{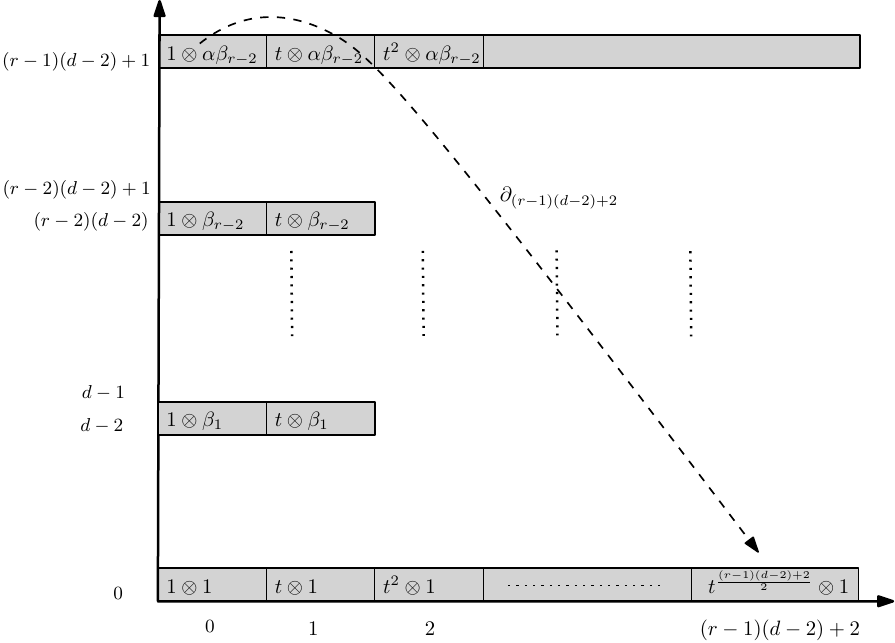}
\caption{\small For $d\geq 4$ even the differential $\partial_{(r-1)(d-2)+2}$ of the Serre spectral sequence \eqref{eq : E2 term}.}\label{fig : 06}
\end{figure}
Hence, the $E_{\infty}$-term is given by
\begin{gather}
E^{p,q}_{\infty}\big(\E(\Z_r)\times_{\Z_r} \GG\big(S^{d-1},r\big)\big)\nonumber\\
\qquad{} \cong
\begin{cases}
	H^p(\Z_r;\F_r), & q=0, \, 0\leq p \leq (r-1)(d-2)+1, \\
	\F_r, & p\in\{0,1\},\, q\in\{(d-2),\ldots, (r-2)(d-2)\},\\
	0, & \text{otherwise},\label{eq : E_{infty} -- 1}
\end{cases}	
\end{gather}
and illustrated in Fig.~\ref{fig : 07}.
Moreover, we have obtained that the map $\pi_G^*$ induced by the projection map of the fibration \eqref{eq : fibration} in cohomology
\begin{equation}
\label{eq : isomorphism -- 1}
	\pi_G^*\colon \ H^p(\Z_r;\F_r)\longrightarrow H^p\big(\E(\Z_r)\times_{\Z_r} \GG\big(S^{d-1},r\big);\F_r\big)
\end{equation}
is an injection for $0\leq p\leq (r-1)(d-2)+1$.
\begin{figure}[h!]\centering
\includegraphics[scale=0.8]{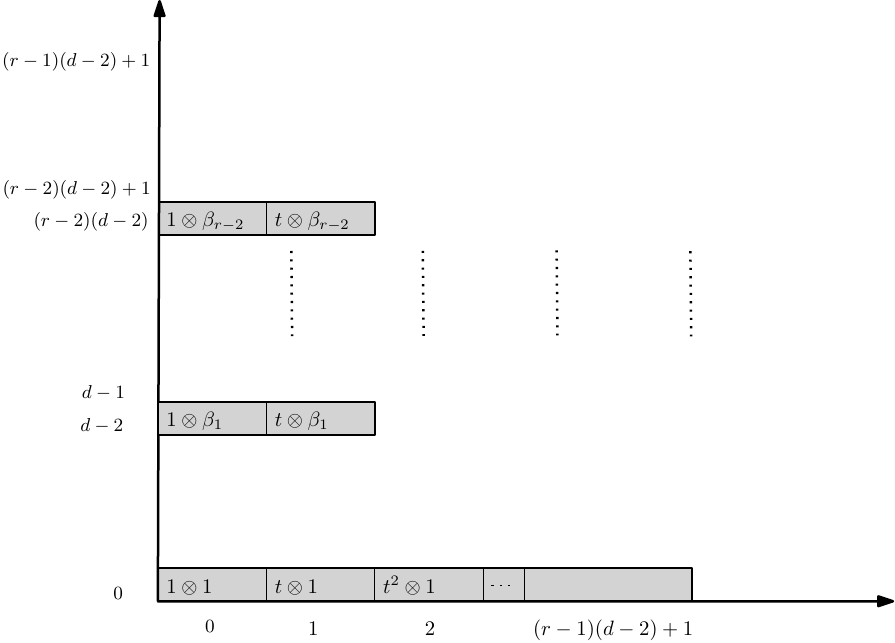}
\caption{\small For $d\geq 4$ even the $E_{\infty}$-term of the Serre spectral sequence \eqref{eq : E2 term}.}\label{fig : 07}
\end{figure}

\subsubsection[$d$ is an odd integer, $d\geq 3$]{$\boldsymbol{d}$ is an odd integer, $\boldsymbol{d\geq 3}$}
\label{subsub -- 02}

The Serre spectral sequence $E^{*,*}_{*}\big(\E(\Z_r)\times_{\Z_r} \GG\big(S^{d-1},r\big)\big)$ has a multiplicative structure. Hence, to compute it completely, we will determine values of the differentials on the generators of the cohomology ring of the fiber.
In this particular case, we need to evaluate the differentials on the elements $\gamma,\delta_1,\ldots,\delta_{\frac{r-3}2}$.
Observe that
\[
(1\otimes\delta_1)^i=1\otimes\delta_1^i=a_i\cdot(1\otimes\delta_i)
\]
for $1\leq i\leq \frac{r-3}2$ and some $a_i\in\F_r{\setminus}\{0\}$.

Similarly to Section \ref{subsub -- 01}, we consider first the values of the differential $\partial_2$ on the generators.
Since the generator $\delta_1$ is of degree $2$ and $H^1\big(\GG\big(S^{d-1},r\big);\F_r\big)=0$ we have that $\partial_2(1\otimes\delta_1)=0$.
Further, the Leibniz rule implies that
\[
\partial_2 (1\otimes\delta_i)= a_i\cdot \partial_2 \big( (1\otimes\delta_1)^i\big)=(ia_i)\cdot\partial_2(1\otimes\delta_1)=0
\]
for $1\leq i\leq \tfrac{r-3}2$ and some $a_i\in\F_r{\setminus}\{0\}$.
Thus, in order to determine the second differential $\partial_2$, we need to determine
\[
\partial_2(1\otimes\gamma)\in E^{2,2(d-2)}_2\big(\E(\Z_r)\times_{\Z_r} \GG\big(S^{d-1},r\big)\big).
\]

Assume that $\partial_2(1\otimes\gamma)=0$.
Then the multiplicative property of the Serre spectral sequence implies that the second differential $\partial_2$ gives
\[
E^{*,*}_{2}\big(\E(\Z_r)\times_{\Z_r} \GG\big(S^{d-1},r\big)\big)\cong E^{*,*}_{3}\big(\E(\Z_r)\times_{\Z_r} \GG\big(S^{d-1},r\big)\big).
\]
The vanishing of cohomology groups $H^*\big(\GG\big(S^{d-1},r\big);\F_r\big)$ in appropriate dimensions implies that the next possible non-zero differential is $\partial_{2(d-2)+1}$.
More precisely, $\partial_{2(d-2)+1}(1\otimes\delta_1)$ might be non-zero, while we know that $\partial_{2(d-2)+1}(1\otimes\gamma)=0$.
Since $2(d-2)+1\leq (r-1)(d-2)$ from Corollary~\ref{cor : differentials in zero row}, we have that $\partial_{2(d-2)+1}(1\otimes\delta_1)=0$.
Consequently, the differential $\partial_{2(d-2)+1}$ vanishes.
About the next differential we know that $\partial_{2(d-2)+2}(1\otimes\delta_1)=0$.
Since $2(d-2)+2\leq (r-1)(d-2)$, again Corollary~\ref{cor : differentials in zero row} implies that $\partial_{2(d-2)+2}(1\otimes\gamma)=0$. Therefore, the differential $\partial_{2(d-2)+2}$ also vanishes.
This means that
\[
E^{*,*}_{2}\big(\E(\Z_r)\times_{\Z_r} \GG\big(S^{d-1},r\big)\big)\cong \cdots \cong E^{*,*}_{2(d-2)+3}\big(\E(\Z_r)\times_{\Z_r} \GG\big(S^{d-1},r\big)\big).
\]
For $s\geq 2(d-2)+3$, all differentials $\partial_s(1\otimes\delta_1)$ and $\partial_s(1\otimes\gamma)$ are zero.
Hence, the multiplicative property of the Serre spectral sequence implies that, for $s\geq 2(d-2)+3$, all the differentials $\partial_s$ vanish.
Thus
\[
E^{*,*}_{2}\big(\E(\Z_r)\times_{\Z_r} \GG\big(S^{d-1},r\big)\big)\cong E^{*,*}_{\infty}\big(\E(\Z_r)\times_{\Z_r} \GG\big(S^{d-1},r\big)\big),
\]
which is a contradiction to the fact, observed after Corollary~\ref{cor : differentials in zero row}, that at least one of the differentials
\[
\partial_{s}\colon \
E_{s}^{p,s-1}\big(\E(\Z_r)\times _{\Z_r}\GG\big(S^{d-1},r\big)\big)\longrightarrow E_{s}^{p+s,0}\big(\E(\Z_r)\times_{\Z_r}\GG\big(S^{d-1},r\big)\big),
\]
for $s\geq (r-1)(d-2)+1$, does not vanish. Hence, $\partial_2(1\otimes\gamma)\neq 0$.

Now set $\partial_2(1\otimes\gamma)=a\cdot (t\otimes\delta_1)$ where $a\in\F_r{\setminus}\{0\}$.
The multiplicative property yields
\begin{align*}
\partial_2(1\otimes \gamma\delta_i) &=
\partial_2((1\otimes \gamma) (1\otimes \delta_i))
= \partial_2(1\otimes \gamma) (1\otimes \delta_i)+(-1)^{d-1}(1\otimes \gamma)\partial_2(1\otimes \delta_i)\\
&= \partial_2(1\otimes \gamma) (1\otimes \delta_i)= a\cdot (t\otimes\gamma_1) (1\otimes \gamma_i)\\
&= \begin{cases}
	\left(a\frac{(i+1)!}{i!1!}\right)\cdot (t\otimes\gamma_{i+1}), & 1\leq i\leq \tfrac{r-3}2,\\
	0, & i=r-2.
\end{cases}
\end{align*}
Consequently, the $E_3$-term is
\begin{gather*}
E^{p,q}_{3}\big(\E(\Z_r)\times_{\Z_r} \GG\big(S^{d-1},r\big)\big)\\
\qquad{} \cong
\begin{cases}
	H^p(\Z_r;\F_r), & q\in\{0,(r-1)(d-2)+1\}, \\
	\F_r, & p\in\{0,1\},\,q\in\big\{2(d-2),\ldots, \tfrac{(r-3)}{2}2(d-2)\big\},\\
	0, & \text{otherwise}.
\end{cases}	
\end{gather*}
Further, the multiplicative property implies that all the differentials $\partial_3,\ldots,\partial_{(r-1)(d-2)+1}$ vanish. Thus,
\[
E^{*,*}_{3}\big(\E(\Z_r)\times_{\Z_r} \GG\big(S^{d-1},r\big)\big)\cong  \cdots  \cong E^{*,*}_{(r-1)(d-2)+2}\big(\E(\Z_r)\times_{\Z_r} \GG\big(S^{d-1},r\big)\big).
\]

The differential $\partial_{(r-1)(d-2)+2}$ is the last remaining differential that can be non-zero.
Since for all~$p$
\[
E^{p,0}_{2}\big(\E(\Z_r)\times_{\Z_r} \GG\big(S^{d-1},r\big)\big)\cong \cdots  \cong E^{p,0}_{(r-1)(d-2)+2}\big(\E(\Z_r)\times_{\Z_r} \GG\big(S^{d-1},r\big)\big),
\]
and we know that
\[
E_{2}^{p,0}\big(\E(\Z_r)\times _{\Z_r}\GG\big(S^{d-1},r\big)\big)\neq 0 = E_{\infty}^{p,0}\big(\E(\Z_r)\times _{\Z_r}\GG\big(S^{d-1},r\big)\big)
\]
for $p\geq r(d-1)$, we have that indeed $\partial_{(r-1)(d-2)+2}\neq 0$.
The multiplicative property again yields that $\partial_{(r-1)(d-2)+2}$ is determined by its image at $1\otimes\gamma\delta_{\frac{r-3}2}$ and has to be non-zero.
Set
\[
\partial_{(r-1)(d-2)+2}\big(1\otimes\gamma\delta_{\frac{r-3}2}\big)=b\cdot\big(t^{\frac{(r-1)(d-2)+2}2}\otimes 1\big)
\]
for some $b\in\F_r{\setminus}\{0\}$. Thus, the homomorphisms
\begin{gather*}
\partial_{(r-1)(d-2)+2}\colon \
E_{s}^{p,(r-1)(d-2)+1}\big(\E(\Z_r)\times _{\Z_r}\GG\big(S^{d-1},r\big)\big)\\
\hphantom{\partial_{(r-1)(d-2)+2}\colon}{} \ \longrightarrow E_{s}^{p+(r-1)(d-2)+2,0}\big(\E(\Z_r)\times_{\Z_r}\GG\big(S^{d-1},r\big)\big)
\end{gather*}
are isomorphism for all $p\geq 0$. Hence, the $E_{\infty}$-term is given by
\begin{gather}
E^{p,q}_{\infty}\big(\E(\Z_r)\times_{\Z_r} \GG\big(S^{d-1},r\big)\big)\nonumber\\
\qquad{} \cong
\begin{cases}
	H^p(\Z_r;\F_r), & q=0,\,0\leq p \leq (r-1)(d-2)+1, \\
	\F_r, & p\in\{0,1\},\,q\in\big\{2(d-2),\ldots, \tfrac{(r-3)}{2}2(d-2)\big\},\\
	0, & \text{otherwise}.
\end{cases}	\label{eq : E_{infty} -- 2}
\end{gather}
Furthermore, we have obtained again that the map $\pi_G^*$, induced by the projection map of the fibration \eqref{eq : fibration} in cohomology
\begin{equation}\label{eq : isomorphism -- 2}
	\pi_G^*\colon \ H^p(\Z_r;\F_r)\longrightarrow H^p\big(\E(\Z_r)\times_{\Z_r} \GG\big(S^{d-1},r\big);\F_r\big),
\end{equation}
is an injection for $0\leq p\leq (r-1)(d-2)+1$.

\subsection{Proof of Theorem \ref{th : topological--main--1}}

Recall that for every $n\in\N$ we have isomorphisms of vector spaces
\[
H^n\big(\GG\big(S^{d-1},r\big)/\Z_r;\F_r\big)\cong H^n\big(\E(\Z_r)\times_{\Z_r} \GG\big(S^{d-1},r\big);\F_r\big)\cong \bigoplus_{p+q=n}E^{p,q}_{\infty}.
\]
Since $H^p(\Z_r;\F_r) = \F_r$ for all $p$, the relations~\eqref{eq : E_{infty} -- 1} for the case of $d\geq 4$ even, and~\eqref{eq : E_{infty} -- 2} for the case of $d\geq 3$ odd, imply the statement of the theorem.

\subsection{Proof of Theorem \ref{th : topological--main--2}}\label{proof of Thm. 3.2}

In order to use the results of Section \ref{sec : Computing the spectral sequence} for a proof of an estimate on the Lusternik--Schnirelmann category of the unlabeled cyclic configuration space $\GG\big(S^{d-1},r\big)/\Z_r$, we first recall some basic notions and results concerning the Lusternik--Schnirelmann category.

The \emph{Lusternik--Schnirelmann category} of a topological space $X$, denoted by $\cat(X)$, is the smallest integer $k$ for which the space $X$ can be covered by $k+1$ open subsets $U_1, U_2,\ldots, U_{k+1}$ with the property that all inclusions $U_i \longrightarrow X$ are nullhomotopic.
Some key properties of the Lusternik--Schnirelmann category that we will use are given in the next lemma, see, e.g.,~\cite{Cornea2003}.

\begin{Lemma}\label{lem:LS-1}\qquad
\begin{enumerate}\itemsep=0pt
\item[\rm (1)] If $X$ is homotopy equivalent to $Y$, then $\cat(X)=\cat(Y)$.
\item[\rm (2)] If $p\colon X\longrightarrow Y$ is a covering, then $\cat(X)\leq\cat(Y)$.
\item[\rm (3)] If $X$ is a $(k-1)$-connected $CW$-complex, then $\cat(X)\leq\tfrac{1}{k}\dim(X)$.
\end{enumerate}
\end{Lemma}

Let $X$ be a topological space, and let $R$ be a commutative ring with unit.
We need the notion of the \emph{category weight} of an element $u\in H^*(X;R)$. Originally introduced by Fadell and Husseini~\cite{Fadell1992}, we use the homotopy invariant version of this notion due to Rudyak~\cite{Rudyak1990} and Strom~\cite{Strom1997}. See \cite[Section~2.7, p.~62; Section~8.3, p.~240]{Cornea2003} for details.

In the next lemma we list properties of the category weight that we will use in the proof of Theorem~\ref{th : topological--main--2}.
For more details on this lemma consult for example \cite[Proposition~8.22, pp.~242--243, p.~259]{Cornea2003} and \cite[Proposition~2.2(3)]{Roth2008}.

\begin{Lemma}\label{lem:LS-2}Let $R$ be a commutative ring with unit.
\begin{enumerate}\itemsep=0pt
\item[\rm (1)] If $0\neq u\in H^{\ell}(X;R)$, then $\wgt(u)\leq\cat(X)$.
\item[\rm (2)] Let $f\colon X\longrightarrow Y$ be a continuous map, and let $u\in H^{\ell}(Y;R)$.
 If $0\neq f^*(u)\in H^{\ell}(X;R)$, then $\wgt(u)\leq\wgt(f^*(u))$.
\item[\rm (3)] If $G$ is a finite group and $0\neq u\in H^{\ell}(\B G;R)$, then $\ell = \wgt(u)$.
\end{enumerate}
\end{Lemma}

In the following we combine previously established results to give a proof of Theorem~\ref{th : topological--main--2}.
Recall that the group $\Z_r$ acts freely on the cyclic configuration space $\GG\big(S^{d-1},r\big)$ and consider the following diagram, which commutes up to a homotopy
\begin{equation}	\label{diag : up to homotopy}
	\xymatrix@1{
\E(\Z_r)\times_{\Z_r} \GG\big(S^{d-1},r\big)\ar[r]^-{p} \ar@/^1.5pc/[rr]^-{\pi_G} & \GG\big(S^{d-1},r\big)/\Z_r\ar[r]^-{c} & \B(\Z_r),
}
\end{equation}
where the map $p$ is induced by the projection on the second factor $\E(\Z_r)\times \GG\big(S^{d-1},r\big)\longrightarrow \GG\big(S^{d-1},r\big)$ and is a homotopy equivalence, the map $\pi_G$ is the projection in the Borel construction fibration \eqref{eq : fibration} induced by the projection on the first factor $\E(\Z_r)\times \GG\big(S^{d-1},r\big)\longrightarrow \E(\Z_r)$, and $c$ is a classifying map associated to the free $\Z_r$ action on $\GG\big(S^{d-1},r\big)$.
Uniqueness of the classifying map up to a homotopy implies that the diagram \eqref{diag : up to homotopy} commutes up to a homotopy; for background see for example \cite[Section~II.1]{Adem2004}.

Applying the cohomology functor $H^*(\,\cdot\,;\F_r)$ to the diagram \eqref{diag : up to homotopy}, we get the following diagram of abelian groups that commutes up to an isomorphism
\begin{gather}
\xymatrix@1{
H^*\big(\E(\Z_r)\times_{\Z_r} \GG\big(S^{d-1},r\big);\F_r\big) &
\ar[l]_-{p^*} H^*\big( \GG\big(S^{d-1},r\big)/\Z_r\big) &
\ar[l]_-{c^*}\ar@/_1.5pc/[ll]_-{\pi_G^*}H^*(\B(\Z_r);\F_r)}\nonumber\\
\qquad{}\cong H^*(\Z_r;\F_r) . \label{diag : up to an isomorphism}
\end{gather}
Let
\[
u\coloneqq t^{\frac{(r-1)(d-2)}{2}}e\in H^{(r-1)(d-2)+1}(\Z_r;\F_r).
\]
Then, according to \eqref{eq : isomorphism -- 1} and \eqref{eq : isomorphism -- 2}, we have that $\pi_G^*(u)\neq 0$. Consequently, from diagram~\eqref{diag : up to an isomorphism}, we get $c^*(u)\neq 0$.
Now Lemma~\ref{lem:LS-2} yields that
\begin{gather*}
\wgt(u)=(r-1)(d-2)+1,\qquad
\wgt(u)\leq\wgt(c^*(u)),\\
\wgt(c^*(u))\leq \cat\big(\GG\big(S^{d-1},r\big)/\Z_r\big).
\end{gather*}
Hence,
\[
\cat\big(\GG\big(S^{d-1},r\big)/\Z_r\big)\geq (r-1)(d-2)+1,
\]
and the proof of Theorem~\ref{th : topological--main--2} is complete.\qed

\section{Proofs of the main results}

According to Section \ref{sec:morse}, the number of $\Z_r$-orbits of the critical points of the length function on the cyclic configuration space $\GG(M,r)$ is bounded below by the category of the space $\GG\big(S^{d-1},r\big)/\Z_r$ (the Lusternik--Schnirelmann theory) or, in general position, by the sum of Betti numbers of $\GG\big(S^{d-1},r\big)/\Z_r$ (the Morse theory).
Thus item~(A) of Theorem~\ref{mainthm} follows from Theorem~\ref{th : topological--main--2}, and item (B) follows from Theorem~\ref{th : topological--main--1}.

\subsection*{Acknowledgements}
We are grateful to Sergei Ivanov for useful discussions on Finsler geometry, and we are grateful to the following sources of funding.
Pavle V. M. Blagojevi\'c, Serge Tabachnikov, and G\"unter M. Ziegler were supported by the DFG via the Collaborative Research Center TRR~109 ``Discretization in Geometry and Dynamics''.
Pavle V. M. Blagojevi\'c was supported by the grant ON 174024 of Serbian Ministry of Education and Science.
Michael Harrison and Serge Tabachnikov were supported by the NSF grant DMS-1510055.
We are also grateful to the referees for their suggestions.

\pdfbookmark[1]{References}{ref}
\LastPageEnding

\end{document}